%% file: ABP23-ArXiV.tex
\documentclass[11pt]{article}

\usepackage[pages=all, color=black, position={current page.south}, placement=bottom, scale=1, opacity=1, vshift=5mm]{background}

\usepackage[margin=1in]{geometry} 

\usepackage{amsmath}
\usepackage{amsthm}
\usepackage{amssymb}
\usepackage{float}
\usepackage{multirow}
\usepackage[caption=false]{subfig}
\usepackage[utf8]{inputenc}
\usepackage{hyperref}
\hypersetup{
	unicode,
	pdfauthor={Author One, Author Two, Author Three},
	pdftitle={A simple article template},
	pdfsubject={A simple article template},
	pdfkeywords={article, template, simple},
	pdfproducer={LaTeX},
	pdfcreator={pdflatex}
}

\usepackage[sort&compress,numbers,square]{natbib}
\usepackage{oubraces}
\usepackage{stix}
\usepackage{stmaryrd}
\theoremstyle{plain}
\newtheorem{theorem}{Theorem}[section]
\newtheorem{corollary}[theorem]{Corollary}
\newtheorem{lemma}[theorem]{Lemma}

\theoremstyle{definition}

\newtheorem{example}{Example}
\newtheorem{Remark}[theorem]{Remark}

\newcommand\bx{\boldsymbol{x}}
\newcommand\bbeta{{\boldsymbol{\beta}}}
\newcommand\calP{{\mathcal{P}}}
\newcommand\calF{\mathcal{F}}
\newcommand\frakh{\mathfrak{h}}

\usepackage{graphicx, color}
\graphicspath{{fig/}}

\usepackage{algorithm, algpseudocode} 
\usepackage{mathrsfs} 

\usepackage{lipsum}

\title{A nodally bound-preserving finite element method for
  reaction-convection-diffusion equations}

\author{Abdolreza  Amiri$^1$ \and Gabriel R. Barrenechea$^1$ \and Tristan Pryer$^2$}

\date{
  $^1$ Department of Mathematics and Statistics, University of Strathclyde, 26
  Richmond Street, Glasgow G1 1XH, Scotland \\ \texttt{\{abdolreza.amiri, gabriel.barrenechea\}@strath.ac.uk}\\%
  $^2$Department of Mathematical Sciences, University of Bath, Claverton down,
  Bath BA2 7AY, UK \\ \texttt{tmp38@bath.ac.uk}\\[2ex]%
}

\begin{document}
\maketitle

\begin{abstract}
  This paper introduces a novel approach to approximate a broad range
  of reaction-convection-diffusion equations using conforming finite
  element methods while providing a discrete solution respecting the 
  physical bounds given by the underlying differential equation. The main result of
  this work demonstrates that the numerical solution achieves accuracy
  of $O(h^k)$ in the energy norm, where $k$ represents the underlying
  polynomial degree. To validate the approach, a series of numerical
  experiments is conducted for various problem instances.  Comparisons
  with the linear continuous interior penalty stabilised method, and
  the algebraic flux-correction scheme (for the piecewise linear
  finite element case) have been carried out, where we can observe
  the favourable performance of the current approach.
  
\end{abstract}

\section{Introduction}
\label{sec:intro}

In numerous fields of mathematical modelling, the partial differential
equations (PDEs) used to model the phenomenon in question respect the
basic laws that were used to derive them, and their solutions satisfy
these {\it first principles}. For example, in the modelling of
incompressible flows, the PDEs involved provide divergence-free
velocities, the PDEs modelling dissipative systems provide solutions
that are energy stable, and in phase-field modelling the solution
has strict global maxima and minima, just to give a few examples.
From a numerical approximation perspective, it is very desirable that
the finite element methods used to discretise these PDEs do respect
the respective laws. Unfortunately, in most instances this is not the
case.  For example, the most popular finite element methods for
incompressible flows are not pointwise divergence-free (see, e.g.,
\cite{JLMNR17} for a recent review on the topic), and proving that
schemes are energy stable is far from trivial (see, e.g.,
\cite{GiesselmannMakridakisPryer:2014,CelledoniJackaman:2021}).

In the particular case of bound-preservation, it was understood very
early that the solution of a standard finite element method does not,
in general, respect the physical bounds.  This was first formalised in
\cite{CR73} in the finite element context where it was shown that the
approximation using piecewise linear finite elements respects such
bounds only if the mesh satisfies certain assumptions about its
internal angles, and how refined it is. In addition to the mesh
restriction, for conservation laws finite element methods that respect
the physical bounds of the problem can either be first order accurate,
or nonlinear, due to the notorious Godunov order barrier theorem (se,
e.g., \cite{God59}). 

As a consequence of the above discussion, in the last few decades
numerous methods that respect global bounds have been proposed.  A
special attention has been given, in fact, to a stronger
property. Namely, numerous methods respecting the {\it Discrete
  Maximum Principle (DMP)} have been proposed, especially for
convection-dominated problems (see \cite{MH85,XZ99,BE05,Kuz07,BJK17},
just to name a few, and \cite{BJK23} for a recent review).  These
methods are, for the most part, nonlinearly stabilised methods. That
is, methods that add a nonlinear stabilising term to the Galerkin
scheme in a way that diffusion is added locally, thus making the
problem locally diffusion-dominated, and avoiding the spurious
oscillations and local violations of the maximum principle.
Interestingly, despite the above-mentioned discretisations being
nonlinear, in most cases the finite element methods proposed are based
on piecewise linear elements. In fact, the extension to higher order
elements imposes even stronger mesh conditions (for example, in
\cite{HM81} it is proven that for the Poisson equation in two space
dimensions a monotone discretisation using quadratic elements can be
built if the mesh is either equilateral, or consists of squares in
which the squares are divided by arbitrary diagonals), and, on the
other hand, there is not much analysis available for nonlinear schemes
using higher order polynomials.

In many situations, the stability of the numerical method does not
require the discrete solution to be devoid of local spurious
oscillations, but only needs to satisfy the global bounds. In such a
case, the problem is actually simpler and there are several
methodologies available. The first option one might think of is simply
cutting off the values that lie outside the admissible range.  This
approach is perhaps used in many simulations without being explicitly
mentioned, and it has been analysed for a linear reaction-diffusion
equation in \cite{Kreuzer14}, and for parabolic problems in
\cite{LHV13}.  A more radical approach consists on {\it reformulating
  the problem} in such a way that the discrete solution always
respects the corresponding bounds (see, e.g.  \cite{HuanChen23} for
the application to Chemotaxis, or \cite{FK04} for an extremely widely
used reformulation in non-Newtonian fluid mechanics).  Alternatively,
global bounds can be enforced by introducing them as inequality
constraints and then approximate a control problem as it has been
done, e.g., in \cite{EHS09}.  Similarly, inequality
constraints can be dealt with using Lagrange multipliers and solving
an extended system, as has been recently done in \cite{ChengShen22},
or \cite{vdVXX19},where a semi-smooth Newton method has also been
introduced to deal with the non-smoothness.

In the very recent work \cite{BGPV23} a different strategy was
followed to impose global bounds in the solution.  The first step is
to define the set, denoted $V_\calP^+$, of {\it admissible} finite
element functions as those satisfying the global bounds {\it at their
  degrees of freedom} (nodal values in the case of Lagrangian
elements); then, introduce an algebraic projection onto the admissible
set, {denote by $u_h^+$ the projection of $u_h^{}$ onto $V_\calP^+$}, 
and write a finite element problem for the projected object.
Since this process introduces a kernel, as the
projection is not injective, a stabilising term is added to remove the
singularity.  This last step allowed, in particular, to avoid the
introduction of Lagrange multipliers.  Moreover, when applied to the
linear reaction-diffusion equation, it turns out that $u_h^+$ is the
orthogonal projection onto $V_\calP^+$, and thus is independent of the
stabilisation used (this last property is lost when applied to more
complicated problems, such as problems with nonlinear reaction, also
addressed in \cite{BGPV23}).  So, the purpose of the present work is
to extend the methodology presented in \cite{BGPV23} to the
convection-diffusion equation. Although the driving principles are
similar, this work presents significant novelities with respect to
\cite{BGPV23}, namely
\begin{enumerate}
\item the starting point in the construction of the method is not the
  plain Galerkin scheme, but a stabilised finite element method
  instead. The reason for this is twofold:  first, in regions where the
  constraint is not active, local oscillations might still appear (and
  linear stabilisation helps with this), and second, our numerical
  experiments show that the addition of linear stabilisation in
  $u_h^+$ helps tremendously the good behaviour of the nonlinear
  solver;
\item the definition of the stabilisation form is different, since we
  need to control the convective term;
\item the analysis differs greatly from that of \cite{BGPV23}.  On the
  one hand, the well-posedness analysis is very different, as the
  discretisation is not driven by a monotone operator. On the other
  hand, since the problem is non-symmetric the solution $u_h^+$ is no
  longer the orthogonal projection of $u$ onto $V_\calP^+$, and thus
  the error analysis follows an alternative path.
\end{enumerate}
	
The remainder of the paper is organised as follows: In \S
\ref{sec:pre} we introduce the notation, the model problem, and all
the preliminary material for the setup of the method (including the
choice of linearly stabilised method to be used).  In
Section~\ref{Sec:FEM} we present the finite element method and show
its well-posedness. The error analysis is carried out in
Section~\ref{Sec:Error}, and in \S \ref{sec:numerics} we test the
performance of the method via numerical experiments, comparing also
with previously existing alterantives. Finally, some conclusions and
future directions are drawn in \S \ref{sec:conc}.
\section{General Setting And The Model Problem}
\label{sec:pre}	
We will adopt standard notations for Sobolev spaces, in line with,
e.g., \cite{EG21-I}. For $D\subseteq\mathbb{R}^{d}$, we denote by
$\|\cdot\|_{0,p,D} $ the $L^{p}(D)$-norm; when $p=2$ the subscript $p$
will be omitted and we only write $\|\cdot\|_{0,D} $. In addition, for
$s\geq 0$, $p\in [1,\infty]$, we denote by $\| \cdot \|_{s,p,D}$ ($|
\cdot |_{s,p,D}$) the norm (seminorm) in $W^{s,p}(D)$; when $p=2$, we
will again omit the subscript $p$ and only write $\|\cdot \|_{s,D}$
($| \cdot |_{s,D}$). In addition, we denote by $H^{-1}(D)$ the dual of
$H_{0}^{1}(D)$ while identifying $L^2(D)$ with its dual. Thus, writing
$\langle\cdot,\cdot\rangle_D^{}$ for the duality pairing, we have
\begin{equation*}
  \langle f,v\rangle_D^{}=\int_D f(\boldsymbol{x})v(\boldsymbol{x})\textrm{d}\boldsymbol{x}\qquad
  \forall\, v\in H^1_0(D)\,,
\end{equation*}
whenever $f\in H^{-1}(D)$ is regular enough.  We do not distinguish
between inner product and duality pairing for scalar or vector-valued
functions.
			
\subsection{The Model Problem} 

Let $\Omega$ be an open bounded Lipschitz domain in $\mathbb{R}^{d}$ ($d=2,3$) with polyhedral boundary $\partial \Omega$.
For a given $f\in H^{-1}(\Omega)$, we consider the following convection-diffusion problem:
\begin{align}
   -{\rm div} (\mathcal{D}\nabla u)+\bbeta\cdot\nabla u+\mu u&= f  \hspace{1cm}{\rm in}\;\Omega, \label{CDR}\\  
   u&= 0  \hspace{1cm}{\rm on}\; \partial \Omega\,,\nonumber
\end{align}
where $\mathcal{D}=(d_{ij})_{i,j=1}^{d}\in L^{\infty}(\Omega)^{d\times d}$, $\bbeta=( \beta_{i})_{i=1}^{d}\in L^{\infty}(\Omega)^{d}$,  and {$\mu\in \mathbb{R}^+$},   
respectively, are the diffusion tensor, the convective field, and the reaction coefficient. We will assume that $ {\rm div} \bbeta=0$, and the diffusion tensor $\mathcal{D}$ is symmetric and uniformly strictly positive definite in $\Omega$;  in other words, there exists a positive constant $d_{0}>0$ such that for almost all $\boldsymbol{x}\in \Omega$, we have  
\begin{align}
  \sum_{i,j=1}^{d}y_{i}^{}d_{ij}(\bx)y_{j}^{}\geq d_{0}\sum_{i=1}^{d}y_{i}^{2} \hspace{1cm}\forall (y_{1},\ldots, y_{d})\in \mathbb{R}^{d}.\label{eq6}
\end{align}

{
\begin{Remark} We have assumed $ {\rm div} \bbeta=0$ only to avoid technical diversions.
The results presented in this paper remind essentially unchanged under the milder condition 
$\mu- {\rm div} \bbeta/2>0$. Moreover, all the results presented below also are applicable to
the case when $\mu\in L^\infty(\Omega)$ is a function that is strictly positive in $\overline{\Omega}$.
\end{Remark}}

The standard weak formulation of  \eqref{CDR} reads as follows: find $u\in H_{0}^{1}(\Omega)$, such that 
\begin{equation}
  a(u,v)=\langle f,v\rangle_{\Omega}\hspace{1cm}\forall v\in H_{0}^{1}(\Omega),\label{eq7}
\end{equation}
where $a(\cdot,\cdot)$ is the bilinear form defined by 
\begin{equation}
  a(w,v):=\left(\mathcal{D}\nabla w,\nabla v\right)_{\Omega}+(\bbeta\cdot\nabla w,v)_{\Omega}+(\mu w,v)_{\Omega}\hspace{1cm}\forall v,w\in H_{0}^{1}(\Omega).\label{eq8}
\end{equation}
The bilinear form  $a(\cdot.\cdot)$ induces the following ``energy'' norm in $H^1_0(\Omega)$
\begin{equation*}
  \|v \|_{a}=\sqrt{a(v,v)}\,,
\end{equation*}
and thus \eqref{eq7} is well-posed thanks to the Lax-Milgram Lemma
(see, e.g. \cite[Lemma~25.2]{EG21-II}).

As it was mentioned in the introduction, our aim is to look for
discrete solutions that respect the same bounds as the solution of
\eqref{eq7}.  Thus, we make the following assumption on $u$, solution
of \eqref{eq7}.

\noindent\underline{Assumption (A1):} We will suppose that the weak solution of \eqref{eq7} satisfies 
\begin{equation}
  0\leq u(\bx)\leq \kappa, \hspace{0.5cm}\text{for almost all}\ \  \bx\in \Omega,\label{eq14}
\end{equation}
where $\kappa$ is a known positive constant.

\begin{Remark} 
Assumption (A1) is, in fact, a re-statement of one of the consequences of the maximum principle for elliptic partial differential equations, see e.g. \cite{evans2010partial}.
The lower bound in \eqref{eq14} is not required to be equal to zero, but we set it as zero for clarity of exposition.  In addition, the same results 
proven in this work hold if $\kappa$ is replaced by a non-negative {continuous} function $\kappa(\bx)$.  {In general, sharp bounds for the constant $\kappa$ are not available, but in some cases
they can be obtained.  For example,  as a consequence of maximum and comparison principles (see 
e.g,  \cite[Corollary~4.4]{renardy2006introduction})
the following bounds can be proven: for almost all $\bx\in \Omega$
the solution $u$ of \eqref{eq7} satisfies 
\begin{equation}
-\frac{\|f\|_{0,\infty,\Omega}^{}}{\mu}\le u(\bx) \le \frac{\|f\|_{0,\infty,\Omega}^{}}{\mu}\,.
\end{equation}
Moreover, if $f\ge 0$ in $\Omega$ we can sharpen the above bound to:
\begin{equation}
0\le u(\bx) \le \frac{\|f\|_{0,\infty,\Omega}^{}}{\mu}\,,
\end{equation}
for almost all $\bx\in\Omega$. Hence, a reasonable estimate for $\kappa$ is $\frac{\|f\|_{0,\infty,\Omega}^{}}{\mu}$
(and this is, in fact, the estimate we have used in our numerical experiments). 
}
\end{Remark}

\subsection{Triangulations, finite element spaces, and preliminary results}	
Let $\calP$ be a conforming, shape-regular, quasi-uniform partition of $\Omega$ into {closed}
simplices (or affine quadrilateral/hexahedra). 
Over $\calP$, and for $k\geq 1$, we define the finite element spaces
{\begin{align}
\tilde{V}_\calP^{}:=&\, \{v_{h}\in C^{0}(\overline{\Omega}):v_{h}^{}|_{K}\in \mathfrak{R}(K)\ \ \  \forall K\in \calP \}\,,\\
  V_\calP^{}:=&\, \tilde{V}_\calP^{}\cap H_{0}^{1}(\Omega)\,,
\end{align}}
where 
\begin{align}
  \mathfrak{R}(K)= \left\{ \begin{array}{ll} \mathbb{P}_{k}(K),  \hspace{1cm}\text{ if $K$ is a simplex},\\  \mathbb{Q}_{k}(K) , \hspace{1cm}\text{ if $K$ is an affine quadrilateral/hexahedral},\end{array} \right.
\end{align}  
with $\mathbb{P}_{k}(K)$ denoting the polynomials of total degree $k$ on $K$ and $\mathbb{Q}_{k}(K)$ denotes the mapped space of polynomial of degree of at most $k$ in each variable.

{\begin{Remark} The results presented below can, in
    principle, be extended to more general quadrilateral meshes.
    Nevertheless, that would require technical diversions due to the
    need to prove norm-equivalences (that are classical for mapped
    elements). To avoid these diversions and keep the presentation
    focused on the bound-preservation aspects, we restrict the
    presentation to affine simplices and quadrilateral/hexahedral
    meshes.
\end{Remark}}

For a mesh $\mathcal{P}$, the following notations are used:
\begin{itemize}
\item let $\{\bx_{1}^{},\bx_{2}^{},\ldots, \bx_{N}^{}\}$ denote the
  set of internal nodes, and the usual Lagrangian basis functions
  associated to these nodes, spanning the space $V_{\calP}^{}$, are
  denoted by $\phi_{1}^{},\ldots,\phi_{N}^{}$;
			
\item let $\calF_{I}^{}$ denote the set of internal facets,
  $\calF_{\partial}^{}$ denotes the set of boundary facets, and
  $\calF_{h}^{}=\calF_{I}^{}\cup \calF_{\partial}^{}$ denotes the set
  of all facets of $\calP$; for an element $K\in\calP$ the set of its
  facets is denoted by $\calF_{K}^{}$;
			
\item for $K\in \calP$, $F\in \calF_{h}^{}$, and a node $\bx_{i}^{}$, we define the following neighbourhoods:
  \begin{align*}
    \omega_{K}^{}&=\bigcup\{K'\in \calP: K\cap K'\neq \emptyset\},
    \\
    \omega_{F}^{}&=\bigcup\{K\in \calP: F\subset K\},
    \\
    \omega_{i}^{}&=\bigcup\{K\in \calP: \boldsymbol{x}_{i}\in K\};
  \end{align*}
\item for a facet $F\in \mathcal{F}_{I}$, $\llbracket \cdot
  \rrbracket$ denotes the jump of a function across $F$.
\end{itemize}

The diameter of a set $G\subset\mathbb{R}^d$ is denoted by $h_G^{}$,
and $h=\max\{h_{K}:K\in \mathcal{P}\}$ stands for the mesh {size}.
We also define the mesh function $\frakh$ as a continuous,
element-wise linear function defined as a local average of meshes
commonly used in finite element analysis \cite[c.f.]{Makridakis:2018}.
{For this, we introduce the set of {\sl vertices} of the mesh,
$\boldsymbol{v}_1^{},\ldots,\boldsymbol{v}_M^{}$,  and 
define $\frakh$ as
the piecewise linear function prescribed by the nodal
values
\begin{equation}\label{mesh-function-definition}
\frakh(\boldsymbol{v}_i^{})=\dfrac{\sum_{K: \boldsymbol{v}_i^{}\in K}h_K^{}}{\# \{K: \boldsymbol{v}_i^{}\in K\}}\,.
\end{equation} }

In the construction of the method, and its analysis, the following
mass-lumped $L^{2}$-inner product will be of importance: for every
polynomial function defined on $K$, in the space $V_\calP^{}$, we
define
\begin{align}
  (v_{h},w_{h})_{h}=\sum_{i=1}^{N}\frakh (\boldsymbol{x}_{i})^{d}v_{h}(\boldsymbol{x}_{i})w_h^{}(\boldsymbol{x}_{i}),
\end{align}  
which induces the norm $\lvert
v_{h}\lvert_{h}=(v_{h},v_{h})_{h}^{\frac{1}{2}}$ in $V_\calP^{}$. This
norm is, in fact, equivalent to the standard $L^2(\Omega)$-norm.  More
precisely, the following result, whose proof can be found in
\cite{EG21-I}, will be used repeatedly in our analysis
below: There exist $C,c>0$, independent of $h$, such that
{
\begin{equation}
  c\,\sum_{i: \bx_i^{}\in K}h_K^d v_h^2(\bx_i^{})\leq \|v_h^{}\|_{0,K}^{2}\leq C\,\sum_{i: \bx_i^{}\in K}h_K^d v_h^2(\bx_i^{})
  \qquad\forall\, K\in \calP,\label{eq5-1}
\end{equation}
and thus, as a consequence of the mesh-regularity  of the mesh, the following holds:}
\begin{equation}
  c\,|v_{h}|_h^{2}\leq \|
  v_h^{}\|_{0,\Omega}^{2}\leq C\,|
  v_h^{}|_h^{2},\label{eq5}
\end{equation}
for all $v_h^{}\in V_\calP^{}$. 

\begin{Remark}
  Let $\mathcal{M}$ denote the mass matrix, whose its entries are
  defined as
  \begin{align*}
    \mathcal{M}_{i,j}=\int_{\Omega}\phi_{j}\phi_{i}\mathrm{d}\bx
    \quad{\rm for }\hspace{0.1cm}i,j=1,\ldots,N.
  \end{align*}
  Then,  setting $v_h^{}=\sum_{i=1}^{N}v_{i}^{}\phi_i^{}$ and $\boldsymbol{v}:=(v_1^{},v_2^{},\ldots,v_N^{})$, 
  as a consequence of \eqref{eq5} and  \cite[Theorem~9.8]{EG21-I}, the following equivalence also holds:
  \begin{equation}
    c\,| v_h^{}|_h^{2}\leq \sum_{i,j=1}^{N}v_i^{}\mathcal{M}_{ij}^{} v_j^{}=\boldsymbol{v}^{T}\mathcal{M}\boldsymbol{v}\leq C\lvert v_{h}\lvert_{h}^{2}\,,\label{eq51}
  \end{equation} 
  where $C,c$ are positive constants independent of $h$ and $v_h^{}$.
\end{Remark}


Next, we recall that the Lagrange interpolation operator is defined by
(see, e.g., \cite[Chapter~11]{EG21-I})
\begin{align}
  i_{h}&:\mathcal{C}^{0}(\overline{\Omega}){\cap H^1_0(\Omega)}\longrightarrow V_{\mathcal{P}},\nonumber\\
  v&\longmapsto i_{h}v=\sum_{j=1}^{N}v(x_{j})\phi_{j}. \label{lagrange}
\end{align}
In addition, the $L^{2}(\Omega)$ -orthogonal projection operator (see
\cite[Chapter~22]{EG21-I}) $\pi:L^{2}(\Omega)\longrightarrow
V_\calP^{}$ is defined as follows
\begin{align}
  \pi: L^{2}(\Omega)&\longrightarrow V_{\mathcal{P}}, \nonumber\\
  w&\longmapsto \pi(w)\;\textrm{where}\;\left(\pi(w),v_h^{}\right)_\Omega^{}=(w,v_h^{})_\Omega^{} \quad\forall\, v_h^{}\in V_\calP^{}\,. \label{def:projection}
\end{align}

With the above ingredients, we now state some inequalities and
properties that will be useful in what follows:
\begin{list}{ }{ }
\item \textbf{a) Inverse inequality:} (\cite[Lemma ~12.1]{EG21-I}) For
  all $m,\ell\in \mathbb{N}_0^{},\,0\le m\leq \ell$ and all $p,q \in
  [1,\infty]$, there exists a constant $C$, independent of $h$, such
  that
  \begin{align}
    | v_h^{}|_{\ell,p,K}^{}\leq Ch_K^{m-\ell+d\left(\frac{1}{p}-\frac{1}{q}\right)}\,| v_h^{}|_{m,q,K}^{}\qquad\forall\, v_h^{}\in V_\calP^{}\,.\label{inverse}
  \end{align}
\item \textbf{b) Discrete Trace inequality:}
  (\cite[Lemma~2.15]{EG21-I}) There exists $C>0$ independent of $h$
  such that, for every $v\in H^{1}(K)$ the following holds
  \begin{equation}
    \|v \|^{2}_{0,\partial K} \leqslant   C \left(h_K^{-1}\|v \|^2_{0,K}+h_K^{}|v |^{2}_{1,K}\right).\label{trace}
  \end{equation}
\item \textbf{c) Approximation property of the Lagrange interpolant:}
  (\cite[Proposition~1.12]{EG21-I}) Let {$1\leq \ell \leq k$} and
  $i_{h}$ be the Lagrange interpolant. Then, there exists $C>0$,
  independent of $h$, for all $h$ and $v\in H^{\ell+1}(\Omega){\cap H^1_0(\Omega)}$ the
  following holds:
  \begin{equation}
    \|v- i_{h}v \|_{0,\Omega}+h |v-i_{h}v |_{1,\Omega}\leq Ch^{\ell +1}|v |_{\ell +1,\Omega}\,.\label{lagranges}
  \end{equation} 
\item \textbf{d) Approximation property of the $L^2$-orthogonal projection
  operator:} (\cite[Section~22.5]{EG21-I}) Let $0\leq \ell \leq k$ and
  $\pi$ be the $L^2(\Omega)$-orthogonal projection. Then, there exists
  $C>0$, independent of $h$, for all $h$ and $v\in H^{\ell+1}(\Omega){\cap H^1_0(\Omega)}$
  the following holds:
  \begin{equation}
    \|v- \pi (v )\|_{0,\Omega}^{} + h\, |v- \pi (v )|_{1,\Omega}^{} \leq Ch^{\ell +1}|v |_{\ell +1,\Omega}^{}\,.\label{proj1}
  \end{equation} 
\end{list}	

\subsection{The algebraic projection onto the admissible set}
With Assumption~(A1) in mind, we define the following subset of finite
element functions that satisfy the bound \eqref{eq14} at the degrees
of freedom:
\begin{align}
  V_{\mathcal{P}}^{+}:=\{v_{h}\in V_{\mathcal{P}}: v_h^{}(\bx_{i})\in[0,\kappa]\ \  \text{for all}\ \ i=1,\ldots,N  \}.\label{eq16}
\end{align} 
Every element $v_{h}\in V_\calP$ can be split as the sum
$v_{h}=v_h^{+}+v_h^{-}$, where $v_h^{+}$ and $v_h^{-}$ are given by
\begin{align}
  v_{h}^{+}=\sum_{i=1}^{N}\max\Big\{0,\min\{v_h(\bx_i),\kappa\}\Big\}\,\phi_i\,,\label{eq17}
\end{align}
and 
\begin{align}
  v_h^{-}=v_h-v_h^{+}. \label{eq18}
\end{align}
We refer to $v_{h}^{+}$ and $v_{h}^{-}$ as the \textit{constrained}
and \textit{complementary} parts of $v_{h}$, respectively.  Using this
decomposition we define the following algebraic projection
\begin{equation}
  (\cdot)^{+}:V_{\mathcal{P}}\rightarrow V_{\mathcal{P}}^{+}\quad,\quad
  v_{h}\rightarrow v_{h}^{+}\,. \label{posoperator}
\end{equation}

 {
 \begin{Remark} If $\kappa$ is not a constant value, but a non-negative continuous
 function, the only difference in the definition of the projection is that in such a case
 the constrained part is given by
 \begin{equation*}
  v_{h}^{+}=\sum_{i=1}^{N}\max\Big\{0,\min\{v_h(\bx_i),\kappa(\bx_i^{})\}\Big\}\,\phi_i\,.
\end{equation*}
 To avoid technical diversions we will not detail such a case, but the results proven in this paper
 remain unchanged.
 \end{Remark}}
 
 The following results concerning this projection will be used repeatedly.
\begin{lemma}\label{lem32}
  Let the operator $(\cdot)^{+}$ be defined in \eqref{posoperator}.
  There exists a constant $C>0$, independent of $h$, such that 
  \begin{align}
    \|w^{+}_{h}-v^{+}_{h} \|_{0,\Omega}^{}&\leq C \|w_{h}-v_{h} \|_{0,\Omega},\label{Lip1}\\
    \|v_{h}^{+} \|_{0,\Omega}&\leq C\kappa,		\label{bound}
  \end{align}
  for all $w_{h},v_{h}\in V_{\mathcal{P}}$.
\end{lemma}
\begin{proof}
{During this proof we drop the subindex $h$ to lighten the notation. Let $w,v\in V_\calP^{}$.
  We start noticing that if $w(\bx_i)\leq v(\bx_i)$, then
  $v^{+}(\bx_i)-w^{+}(\bx_i) \leq v(\bx_i)-w(\bx_i)$, and when
  $v(\bx_i)\leq w(\bx_i) $, we have $v^{+}(\bx_i)-w^{+}(\bx_i) \leq
  -(v(\bx_i)-w(\bx_i))$. So,
  $|v^{+}(\boldsymbol{x}_{i})-w^{+}(\boldsymbol{x}_{i})| \leq |
  v(\boldsymbol{x}_{i})-w(\boldsymbol{x}_{i})|$.  Now, set $e_{i}:=|
  v(\boldsymbol{x}_{i})-w(\boldsymbol{x}_{i})|$, and $\boldsymbol
  e=(e_{j})_{j=1}^{N}$. Then, using \eqref{eq51} we obtain
  \begin{align*}
    \|v^{+}(\boldsymbol{x}_{i})-w^{+}(\boldsymbol{x}_{i})\|_{0,\Omega}^{2}&=\int_{\Omega}\left(\sum_{i=1}^{N} (v^{+}(\boldsymbol{x}_{i})-w^{+}(\boldsymbol{x}_{i}))\phi_{i}\right)^{2}\mathrm{d}\bx \leq \int_{\Omega}\left(\sum_{i=1}^{N}e_{i}\phi_{i} \right)^{2}\mathrm{d}\bx\\
    &=\sum_{i,j=1}^{N} e_{j}\mathcal{M}_{ij}e_{i}\leq C\|v-w \|_{0,\Omega}^{2},
  \end{align*} 
  which proves \eqref{Lip1}.
  To prove \eqref{bound} we use \eqref{eq5} and the mesh regularity to obtain that, for all $v\in V_\calP^{}$, we have
  \begin{align*}
    \| v^{+}\|_{0,\Omega}^{2}\leq C\sum_{i=1}^{N}\frakh (\boldsymbol{x}_{i})^{d}(v^{+}\left(\boldsymbol{x}_{i})\right)^{2}\leq C\left(\sum_{i=1}^{N}\frakh (\boldsymbol{x}_{i})^{d}\right) \kappa^{2}\leq C\kappa^{2},
  \end{align*}
  concluding the proof.}
\end{proof}	
		
\subsection{A linear stabilised method}

As it was mentioned in the introduction, in the method's definition we
need to introduce a linear stabilising term aimed at dampening the
oscillations caused by the dominating convection.  In this work we
have chosen to use the Continuous Interior Penalty (CIP) method
originally proposed in \cite{Burman}. This method adds the following
stabilising term to the Galerkin scheme:
{\begin{align}
  J(u_h^{},v_h^{})=\gamma \sum_{F\in\calF_I^{}}\int_F\|\bbeta \|_{0,\infty,F}^{} h_{F}^{2}\,\llbracket \nabla u_{h} \rrbracket \cdot \llbracket \nabla v_{h} \rrbracket \,\mathrm{d}s\,.
\label{eq10}
\end{align}}
Here,  
$\gamma\ge 0$ is a non-dimensional constant.  Using this stabilising
term, the CIP stabilised method proposed in \cite{Burman} reads as
follows: find $u_h\in V_\calP$ such that
\begin{equation}
  a_{J}(u_{h},v_{h}):=a(u_{h},v_{h})+	J(u_{h},v_{h})
  =
  \langle f,v_{h}\rangle_{\Omega}\hspace{1cm}\forall v_h^{}\in V_\calP\,.\label{eq12}
\end{equation}
The bilinear form $a_{J}(\cdot,\cdot)$ induces the following norm on $V_\calP$ 
\begin{align}
  \|v_{h} \|_{h}
  :=
  a_{J}(v_{h},v_{h})^{\frac{1}{2}}
  =
  \left(\|\mathcal{D}^{\frac{1}{2}} \nabla v_{h}\|^{2}_{0,\Omega}
  +
  \lVert \mu^{\frac{1}{2}} v_{h}\lVert_{0,\Omega}^{2}
  +
  J(v_{h},v_{h})\right)^{\frac{1}{2}}.
  \label{norm}
\end{align}
The following result will be of use in the error analysis.

\begin{lemma}\label{lem33}
  There exists $C>0$ such that for $v_{h}\in V_{\mathcal{P}}$, the
  penalty term (\ref{eq10}) satisfies the following property
  \begin{align}
    \frac{\gamma}{\|\bbeta\|_{0,\infty,\Omega}}\left\|h^{\frac{1}{2}}\left(\bbeta\cdot\nabla v_{h} -\pi(\bbeta\cdot\nabla v_{h})\right)\right\|^{2}_{0,\Omega}\leq   C\,J(v_{h},v_{h}).\label{equ32}
  \end{align}  
  Moreover, the stabilising  term \eqref{eq10} satisfies the following bounds: there exists $C>0$, independent of $h$ and any
  physical constant, such that for all  $w_h,v_h\in V_{\mathcal{P}}$ the following holds
  \begin{align}
    J(v_{h},w_{h})&\leq C\gamma\,h\,\|\bbeta\|_{0,\infty,\Omega}|v_{h} |_{1,\Omega}|w_{h} |_{1,\Omega},\label{equ35}\\
    J(v_{h},w_{h})&\leq C\gamma\left(\sum_{K\in \mathcal{P}} h_{ K}^{-1}\|\bbeta \|_{0,\infty,K}\|  v_{h}\|_{0,K}^{2}\right)^{\frac{1}{2}}\left(\sum_{K\in \mathcal{P}} h_{ K}^{-1}\|\bbeta \|_{0,\infty,K}\|  w_{h}\|_{0,K}^{2}\right)^{\frac{1}{2}}. \label{equ36}
  \end{align}  
\end{lemma}
\begin{proof}
  The inequality \eqref{equ32} {is a direct consequence of the result proven in}
  \cite[Lemma~5]{Burman}.  To prove \eqref{equ35} we use the
  Cauchy-Schwarz inequality, the local trace result \eqref{trace} and
  the inverse inequality \eqref{inverse} to obtain
  {\begin{align*}
    J(v_{h},w_{h})&= \sum_{F\in\calF_I^{}}\int_F\gamma \|\bbeta \|_{0,\infty,F} h_F^{2}\llbracket\nabla v_{h}^{} \rrbracket \cdot\llbracket\nabla w_{h} \rrbracket \mathrm{d}s\\
    &\le \left\{ \sum_{F\in\calF_I^{}} \gamma \|\bbeta \|_{0,\infty,F} h_F^{2}
    \|\llbracket\nabla v_{h}^{} \rrbracket\|_{0,F}^{2}\right\}^{\frac{1}{2}}
    \left\{ \sum_{F\in\calF_I^{}} \gamma \|\bbeta \|_{0,\infty,F} h_F^{2}
    \|\llbracket\nabla w_{h} \rrbracket\|_{0,F}^2\right\}^{\frac{1}{2}}\\ 
    &\leq  C\gamma\left(\sum_{K\in \calP}\|\bbeta \|_{0,\infty,K}h_{ K}^{2}\| \nabla \big(v_h^{}|_K^{}\big) \|_{0,\partial K}^{2}\right)^{\frac{1}{2}}\left(\sum_{K\in \calP}\|\bbeta \|_{0,\infty,K}^{}h_{ K}^{2} \|\nabla \big(w_h^{}|_K^{}\big) \|_{0,\partial K}^{2}\right)^{\frac{1}{2}} \\
    &\leq C\gamma \left(\sum_{K\in \mathcal{P}}h_{ K}\|\bbeta \|_{0,\infty,K}^{} \| \nabla v_{h}^{}\|_{0,K}^{2}\right)^{\frac{1}{2}}\left(\sum_{K\in \mathcal{P}} h_{ K}\|\bbeta \|_{0,\infty,K}\| \nabla w_{h}\|_{0,K}^{2}\right)^{\frac{1}{2}}\,,
  \end{align*} }
  which proves \eqref{equ35}.  The proof of \eqref{equ36} follows from
  the last inequality above and one further the application of the
  inverse inequality \eqref{inverse}.
\end{proof}	
	
\begin{Remark}
  An alternative definition of the CIP stabilising term, that we will
  also use in some of our numerical results, is given by penalising
  the upwind gradient jumps rather than the normal gradient, that is,
  \begin{equation}
    {J(u_h,v_h)
    =
    \gamma \sum_{F\in\calF_I^{}}\int_F\frac{\gamma_{\bbeta}}{\|\bbeta
      \|_{0,\infty,F}}  h_{F}^{2}\,\llbracket\bbeta\cdot\nabla
    u_{h}\rrbracket \llbracket\bbeta\cdot\nabla
    v_{h}\rrbracket \textrm{d}s\,,\label{eq121}}
  \end{equation}
  where $\gamma_\bbeta^{}\ge 0$ is a non-dimensional constant. In all
  our proofs we have used the term given by \eqref{eq10}, but the
  proofs remain valid if we use \eqref{eq121} instead.
\end{Remark}

\begin{Remark}
  The choice of CIP stabilisation has been made for convinience and
  simplicity of the presentation. In the numerical experiments we will
  show that the addition of the linear stabilising term has a positive
  effect on the performance of the method, more specifically, it will
  improve the performance of the nonlinear solver greatly.  From a
  stability/error estimates point of view, it is also worth mentioning
  that the exact same results proven in this work are also valid for
  other choices of linear stabilisation, e.g., local projection
  stabilisation \cite{KL09}, or subgrid viscosity
  \cite{Guermond_1999}, for example.
\end{Remark}
\section{The finite element method}\label{Sec:FEM}
				
The finite element method proposed in this work reads as follows: find
$u_h\in V_\calP$ such that
\begin{equation}
  a_{h}(u_{h};v_{h})
  =
  \langle f,v_{h}\rangle_{\Omega}\hspace{1cm}\forall v_{h}\in V_\calP\,,\label{eq19}
\end{equation}	
where the nonlinear form $a_{h}(\cdot;\cdot)$ is defined by
\begin{align}
  a_{h}(u_{h};v_{h}):=a_{J}(u_{h}^{+},v_{h})+s(u_{h}^{-},v_{h})\,.\label{FEMethod}
\end{align}
Here, $a_J(\cdot,\cdot)$ is the bilinear form defined in \eqref{eq12},
$u_h^+$ and $u_h^-$ are defined in \eqref{eq17},\eqref{eq18}. The
bilinear form $s(\cdot,\cdot)$ is added in order to control the
complementary part $u_h^-$, and is defined as follows:
\begin{align}
  s(v_{h},w_{h})=\alpha\sum_{i=1}^{N}\left(\|\mathcal{D} \|_{0,\infty,\omega_{i}}\frakh (\boldsymbol{x}_{i})^{d-2}+\|\bbeta \|_{0,\infty,\omega_{i}}\frakh (\boldsymbol{x}_{i})^{d-1}+ {\mu} 
  \frakh (\boldsymbol{x}_{i})^{d}\right)v_{h}(\boldsymbol{x}_{i})w_{h}(\boldsymbol{x}_{i})\,,\label{eq20}
\end{align}	
where the parameter $\alpha>0$ is a non-dimensional constant.  The
stabilising form $s(\cdot,\cdot)$ induces the following norm in
$V_\calP^{}$:
\begin{align}
  \| v_{h}\|_{s}=\sqrt{s(v_{h},v_{h})}\,.\label{eq21}
\end{align}

The following result, that appears as a consequence of \eqref{eq5},
shows that the stabilising bilinear form $s(\cdot,\cdot)$ indeed
controls $u_h^-$, more specifically it controls the kernel of the
projection $(\cdot)^+$.
\begin{lemma}\label{Lem:s}
  There exists a constant $C_{\rm equiv}^{}>0$, depending only on the shape regularity of $\calP$, such that
  \begin{align}
    \| v_{h}\|_{h}^{2}\leq \frac{C_{\rm equiv}}{\alpha}\|v_{h} \|_{s}^{2} \qquad\forall v_{h}\in V_\calP\,, \label{eq22}
  \end{align}
  where $\| \cdot\|_h$ is the norm defined in \eqref{norm}.
\end{lemma}
\begin{proof}
  Using the inverse inequality \eqref{inverse},  {\eqref{eq5-1}, and the mesh regularity }we obtain
  \begin{align}
    \|\mathcal{D}^{\frac{1}{2}} \nabla v_{h}\|^{2}_{0,\Omega}+\lVert \mu^{\frac{1}{2}} v_{h}\lVert_{0,\Omega}^{2}\leq C\sum_{i=1}^{N}\left(\|\mathcal{D} \|_{0,\infty,\omega_{i}}\frakh (\boldsymbol{x}_{i})^{d-2}+{\mu}
    \,\frakh (\boldsymbol{x}_{i})^{d}\right)v_{h}(\boldsymbol{x}_{i})^{2}.  
  \end{align}
  Also, \eqref{equ36} and \eqref{eq5-1} yield
  \begin{align*}
    J(v_{h},v_{h})&\leq C\gamma\sum_{K\in \mathcal{P}}h_{ K}^{-1}\|\bbeta \|_{0,\infty,K} \|  v_{h}\|_{0,K}^{2}\leq C\gamma\sum_{K\in \calP}h_{ K}^{-1}h_{ K}^{d}\|\bbeta \|_{0,\infty,K}\sum_{\bx_i\in K}  v_h(\bx_i)^{2}\\&\leq C\gamma \sum_{i=1}^{N}\|\bbeta \|_{0,\infty,\omega_{i}}\frakh (\bx_i)^{d-1}v_{h}(\boldsymbol{x}_{i})^{2}.
  \end{align*}  
  Gathering the last two bounds proves \eqref{eq22} with $C_{\rm equiv}=(1+\gamma)C$.
\end{proof}
\begin{Remark}
  { The result of Lemma \ref{Lem:s} explains the
    scaling factors chosen for defining \( s(\cdot,\cdot) \).
    Additionally, our formula \eqref{eq20} differs from those used in
    reaction diffusion equations (as seen in \cite{BGPV23}) by
    including a specific term, \( \|\bbeta
    \|_{0,\infty,\omega_{i}}\frakh (\boldsymbol{x}_{i})^{d-1} \). This
    inclusion is crucial for proving \eqref{eq22}, which is important
    for both ensuring the problem is well-posed and for the error
    analysis. From our practical experience, this term also enhances
    the performance of the nonlinear solver. As for the structure of
    the stabilisation, we opted for a mass-lumped approach in defining
    \( s(\cdot,\cdot) \), largely because of the monotonicity we
    established in Lemma~\ref{l31}.
  }
\end{Remark}	
\subsection{Well-posedness}
In this section, we analyse the existence and uniqueness of solutions
for \eqref{FEMethod}. The first step is given by the following
monotonicity result, whose proof is identical to that of
\cite[Lemma~3.1]{BGPV23}.
\begin{lemma}\label{l31}
  The bilinear form $s(\cdot,\cdot)$ defined in \eqref{eq20} satisfies
  the following inequalities:
  \begin{eqnarray}
    s(v_{h}^{-}-w_{h}^{-},v_{h}^{+}-w_{h}^{+})\geq 0\qquad \forall v_{h},w_{h} \in V_{\mathcal{P}},\hspace{0.8cm}\label{eq23}\\
    s(v_{h}^{-},w_{h}-v_{h}^{+})\leq 0\qquad \forall v_{h} \in V_{\mathcal{P}},w_{h} \in V_{\mathcal{P}}^{+}\label{eq24}.
  \end{eqnarray} 
\end{lemma}	

Despite the fact that $s(\cdot,\cdot)$ is monotone, the discrete
problem \eqref{eq19} is not driven by a monotone nonlinear
mapping. So, the well-posedness of \eqref{FEMethod} needs to be proven
using different arguments to those used in \cite{BGPV23}.  We will
first prove existence that will appear as a consequence of Brouwer's
fixed point theorem, and only after linking any solution $u_h^+$ of
\eqref{eq19} to a variational inequality, we will be able to prove a
uniqueness result for $u_h^{}$.


\begin{theorem}
  \label{Theorem14}
  {Suppose that $\alpha\ge C_{\rm equiv}$. Then, }
  there exists $u_{h}\in V_{\mathcal{P}}$ that solves \eqref{eq19}.
\end{theorem}
\begin{proof}
  We begin by defining the bilinear form
  \begin{align*}
    \tilde{a}_{J}(v_{h},w_{h})
    :=
    \left(\mathcal{D}\nabla v_{h},\nabla w_{h}\right)_{\Omega}
    +
    \mu(v_{h},w_{h})_{\Omega}
    +
    J(v_{h},w_{h})\hspace{1cm}\forall v_{h},w_{h}\in V_{\mathcal{P}},
  \end{align*}
  and the mapping 
  \begin{align*}
    T:&V_{\mathcal{P}}\longrightarrow V_{\mathcal{P}},\\
    \widehat{u}_{h}&\longrightarrow u_{h}=T(\widehat{u}_{h}),
  \end{align*}
  where $u_{h}=T(\widehat{u}_{h})$ solves the following equation
  \begin{align}
    \tilde{a}_{J}(u_{h}^{+},v_{h})+s(u_{h}^{-},v_{h})=\langle f,v_{h}\rangle_{\Omega} -
    (\bbeta \cdot \nabla {\widehat{u}_{h}^+},v_{h})_{\Omega}\qquad\forall v_{h}\in V_{\mathcal{P}}.\label{eq49}
  \end{align}
  We observe that $u_h$ solves \eqref{eq19} if and only if
  $T(u_h)=u_h$. So, the proof will consist on proving that $T$
  satisfies the hypotheses of Brouwer's fixed point Theorem
  \cite[Theorem~10.41]{renardy2006introduction}.

  \noindent i) \underline{$T$ is well-defined:} To prove that $T$ is
  well-defined, we see that \eqref{eq49} is a particular example of
  the method proposed in \cite{BGPV23}. So, using
  \cite[Theorem~3.2]{BGPV23}, there exists a unique solution $u_{h}\in
  V_\calP$ of \eqref{eq49}, and thus $T$ is well-defined.
  
\noindent ii) \underline{ $T$ is continuous:} 
{Since we have supposed that $\alpha$ is large enough, we can use the
monotonicity result proven in \cite[Theorem~3.2]{BGPV23} and obtain that, for all
$v_h,w_h\in V_\calP$
\begin{equation*}
  \tilde{a}_{J}(v_{h}^{+}-w_{h}^{+},v_{h}-w_{h})+s(v_{h}^{-}-w_{h}^{-},v_{h}-w_{h})\geq 
  C\|v_{h}-w_{h} \|_{h}^{2}\,,
\end{equation*}
where $C>0$ is independent of $h$.}
Next, suppose that for $\widehat{v}_{h}, \widehat{w}_{h}\in V_{\mathcal{P}}$
and let $v_{h}=T(\widehat{v}_h)$ and $w_{h}=T(\widehat{w}_h)$.  Integrating by
parts, using {H\"older's} inequality, Lemma \ref{lem32}, and \eqref{eq22}
we obtain
\begin{align*}
  C\|v_{h}-w_{h} \|_{h}^{2}&\leq	\tilde{a}_{J}(v_{h}^{+}-w_{h}^{+},v_{h}-w_{h})+s({v_{h}^{-}-w_{h}^{-}},v_{h}-w_{h})\\                                      
  &= -(\bbeta\cdot \nabla (\widehat{v}_{h}^{+}-\widehat{w}_{h}^{+}),v_{h}-w_{h})_{\Omega}\\
  & =(\widehat{v}_{h}^{+}-\widehat{w}_{h}^{+},\bbeta\cdot \nabla(v_{h}-w_{h}) )_{\Omega}\\
  &\leq C\| \bbeta\|_{0,\infty,\Omega}\|\widehat{v}_{h}-\widehat{w}_{h} \|_{0,\Omega}|v_{h}-w_{h}  |_{1,\Omega}\\&\leq 
  C\|\bbeta\|_{0,\infty,\Omega}\|\widehat{v}_{h}-\widehat{w}_{h} \|_{0,\Omega}\|\mathcal {D}^{-\frac{1}{2}}\|_{0,\infty,\Omega}\|\mathcal {D}^{\frac{1}{2}}\nabla(v_{h}-w_{h} ) \|_{0,\Omega}
  \\&\leq C\frac{\|\bbeta\|_{0,\infty,\Omega}}{d_{0}^{\frac{1}{2}}} \|\widehat{v}_{h}-\widehat{w}_{h}\|_{0,\Omega}\|v_{h}-w_{h} \|_{h}.
\end{align*}
Therefore
\begin{align*}
  \|T(\widehat{v}_{h})-T(\widehat{w}_{h}) \|_{h}\leq C\frac{\|\bbeta\|_{0,\infty,\Omega}}{d_{0}^{\frac{1}{2}}}\|\widehat{v}_{h}-\widehat{w}_{h}\|_{0,\Omega},
\end{align*}
and $T$ is Lipschitz continuous.

\noindent iii) \underline{There exists $R>0$, such that
  $T(B(0,R))\subseteq B(0,R)$:} Let $\widehat{z}_{h}\in V_\calP$ be
arbitrary and $z_{h}=T(\widehat{z}_{h})$. By using $v_{h}=z_{h}^{+}$ in
\eqref{eq49}, Cauchy-Schwarz's and H\"older's inequalities, and \eqref{bound} we get
{
\begin{align*}
  \tilde{a}_{J}(z_{h}^{+},z_{h}^{+})+\underbrace{s(z_{h}^{-},z_{h}^{+})}_{\ge 0} &=\langle f,z_{h}^{+}\rangle_{\Omega} - (\bbeta \cdot \nabla \widehat{z}_{h}^{+},z_{h}^{+})_{\Omega}^{}\\
 & \le \|f \|_{0,\Omega}\|z_{h}^{+} \|_{0,\Omega} 
+\, (\widehat{z}_{h}^{+},\bbeta \cdot \nabla z_{h}^{+})_{\Omega}^{}\\
  &\le C\, \|f \|_{0,\Omega}\,\mu^{-\frac{1}{2}}\|z_h^+\|_h^{}+ 
 \|\bbeta\|_{0,\infty,\Omega}^{}\|\widehat{z}_{h}^{+}\|_{0,\Omega}^{}d_0^{-\frac{1}{2}}
  \|z_h^+\|_h^{}\\
  &\le C\,\left\{ \frac{\|f \|_{0,\Omega}}{\mu^{\frac{1}{2}}} + 
  \frac{\|\bbeta\|_{0,\infty,\Omega}^{}\,\kappa}{d_0^{\frac{1}{2}}}\right\}\, \|z_h^+\|_h^{}\,,
\end{align*}
and so $z_{h}^{+}$ satisfies
\begin{equation}\label{Bound-zh+}
  \|z_{h}^{+} \|_{h}\leq C\,\left\{ \frac{\|f \|_{0,\Omega}}{\mu^{\frac{1}{2}}} + 
  \frac{\|\bbeta\|_{0,\infty,\Omega}^{}\,\kappa}{d_0^{\frac{1}{2}}}\right\}\,.
\end{equation}}
Next, we take $v_{h}=z_{h}^{-}$ in \eqref{eq49}.  Integrating by
parts and using H\"older's inequality we get
\begin{align*}
  \tilde{a}_{J}(z_{h}^{+},z_{h}^{-})+s(z_{h}^{-},z_{h}^{-})
  &=
  \langle f,z_{h}^{-}\rangle_{\Omega} -(\bbeta \cdot \nabla {\widehat{z}_{h}^{+}},z_{h}^{-})_{\Omega}
  \\
  &\leq
  \|f \|_{0,\Omega}\|z_{h}^{-} \|_{0,\Omega}
  +
  \|{\widehat{z}_{h}^{+}} \|_{0,\Omega}\| \bbeta\|_{0,\infty,\Omega}|z_{h}^{-} |_{1,\Omega}.
\end{align*}
Now using \eqref{bound} we have
\begin{align*}
  \tilde{a}_{J}(z_{h}^{+},z_{h}^{-})+s(z_{h}^{-},z_{h}^{-})
  &\leq
  C\left(\|f \|_{0,\Omega}\|z_{h}^{-} \|_{0,\Omega}
  +
  \kappa \| \bbeta\|_{0,\infty,\Omega}|z_{h}^{-} |_{1,\Omega}\right)
  \\&\leq 
  C \left(\frac{\|f\|_{0,\Omega}}{\mu^{{1}/{2}}}
  +
  \kappa\frac{\| \bbeta\|_{0,\infty,\Omega}}{d_{0}^{{1}/{2}}}\right)\|z_{h}^{-} \|_{h},
\end{align*}
and {applying \eqref{eq22} and  Young's inequality we get}
\begin{align*}
  \tilde{a}_{J}(z_{h}^{+},z_{h}^{-})+s(z_{h}^{-},z_{h}^{-})
  &\leq
  C \left(\frac{\|f\|_{0,\Omega}}{\mu^{{1}/{2}}}
  +
  \kappa\frac{\| \bbeta\|_{0,\infty,\Omega}}{d_{0}^{{1}/{2}}}\right)^{2}
  +
  \frac{s(z_{h}^{-},z_{h}^{-})}{2}
  \\
  &=:
  \frac M2
  +
  \frac{s(z_{h}^{-},z_{h}^{-})}{2}.
\end{align*}
Using Young's and Cauchy-Schwarz's inequalities for
$\tilde{a}_{J}(z_{h}^{+},z_{h}^{-})$, and \eqref{Bound-zh+}, yields
\begin{align*}
  -\delta 	\tilde{a}_{J}(z_{h}^{-},z_{h}^{-})+s(z_{h}^{-},z_{h}^{-})\leq M +C \delta^{-1}\tilde{a}_{J}(z_{h}^{+},z_{h}^{+})\leq M+{C\,\, \delta^{-1} \left\{ \frac{\|f \|_{0,\Omega}}{\mu^{\frac{1}{2}}} + 
  \frac{\|\bbeta\|_{0,\infty,\Omega}^{}\,\kappa}{d_0^{\frac{1}{2}}}\right\}}\,,
\end{align*}
for any $\delta>0$. Then, choosing $\delta$ small enough, and
Lemma~\ref{Lem:s} we get
\begin{align*}
  \|z_{h}^{-} \|_{h}
  \leq
  C \big(-\delta	\tilde{a}_{J}(z_{h}^{-},z_{h}^{-})+s(z_{h}^{-},z_{h}^{-}) \big)
  \leq C_{2}(f,\mu,\mathcal {D},\bbeta,\kappa),
\end{align*}
where 
\begin{equation*}{
C_{2}(f,\mu,\mathcal {D},\bbeta,\kappa)=M+ C\, \left\{ \frac{\|f \|_{0,\Omega}}{\mu^{\frac{1}{2}}} + 
  \frac{\|\bbeta\|_{0,\infty,\Omega}^{}\,\kappa}{d_0^{\frac{1}{2}}}\right\}\,.}
  \end{equation*}
  Hence, $z_h=T(\widehat{z}_h)$ satisfies the
following (uniform) bound
\begin{align*}
  \|z_{h} \|_{h}&\leq \|z_{h}^{-} \|_{h} +\|z_{h}^{+} \|_{h} \leq  
C\, \left\{ \frac{\|f \|_{0,\Omega}^{}}{\mu^{\frac{1}{2}}} + 
  \frac{\|\bbeta\|_{0,\infty,\Omega}^{}\,\kappa}{d_0^{\frac{1}{2}}}\right\}
    +C_{2}^{}(f,\mu,\mathcal {D},\bbeta,\kappa)=:R.
\end{align*}
Therefore, $z_{h}=T(\widehat{z}_{h})\in B(0,R)$, for every $\widehat{z}_{h}\in
V_{\mathcal{P}}$, which shows that $T(B(0,R))\subseteq B(0,R)$.

Hence, using Brouwer's fixed point theorem, there exists $u_{h}\in
V_{\mathcal{P}}$ such that $T(u_{h})=u_{h}$. In other words, problem
\eqref{eq19} has at least one solution.
\end{proof}


The proof of the last result does not imply uniqueness of
solutions. The next two results will close that gap, whilst at the
same time providing a very useful characterisation for $u_h^+$.
 
\begin{lemma}\label{lemma5}
  Let $u_{h}\in V_\calP$ solve \eqref{eq19}. Then, $u_{h}^{+}\in V_\calP^+$ satisfies
  \begin{align}
    a_{J}(u_{h}^{+},v_{h}-u_{h}^{+})\geq \langle f,v_{h}-u_{h}^{+}\rangle_{\Omega}\qquad\forall v_{h}\in V_{\mathcal{P}}^{+},\label{variational}
  \end{align}
  where $a_{J}(\cdot,\cdot)$ is defined in (\ref{eq12}). In addition, $u^{-}_{h}$ is the unique solution of
  \begin{align}
    s(u_{h}^{-},v_{h})=\langle f,v_{h}\rangle-a_{J}(u_{h}^{+},v_{h}) \qquad\forall v_{h}\in V_{\mathcal{P}}.\label{Sunique}
  \end{align}  
\end{lemma}
\begin{proof}
  Testing \eqref{eq19} with $v_{h}^{}\in V_\calP^+$ and $u_h^+$  as test functions gives
  \begin{align*}
    a_{J}(u_{h}^{+},v_h)+s(u_{h}^{-},v_h)= \langle f,v_h\rangle_{\Omega},\\
    a_{J}(u_{h}^{+},u_{h}^{+})+s(u_{h}^{-},u_{h}^{+})= \langle f,u_{h}^{+}\rangle_{\Omega}.
  \end{align*} 
  Subtracting the second equation from the first one we arrive at
  \begin{align*}
    a_J(u_{h}^{+},v_h-u_{h}^{+})+s(u_{h}^{-},v_h-u_{h}^{+})=\langle f,v_h-u_{h}^{+}\rangle_{\Omega}\hspace{0.5cm}\forall v_{h}\in V_{\mathcal{P}}^{+},
  \end{align*} 
  and then, using (\ref{eq24}) in Lemma \ref{l31} we get that
  $u_{h}^{+}\in V_{\mathcal{P}}^{+}$ satisfies (\ref{variational}).
  Finally, since $s(\cdot,\cdot)$ is an elliptic bilinear form in
  $V_\calP$, $u_h^-$ is the unique solution of \eqref{Sunique}, thus
  concluding the proof.
\end{proof}		

The last result provides a characterisation of any solution of
\eqref{eq19} as the solution of the two successive problems
\eqref{variational}, \eqref{Sunique}.  

\begin{corollary}\label{corollary}
  Problem \eqref{eq19} has a unique solution.	
\end{corollary}
\begin{proof}
  If $u_{1},u_{2}\in V_{\mathcal{P}}$ satisfy (\ref{eq19}), then
  $u_{1}^{+}$ and $u_{2}^{+}$ satisfy (\ref{variational}). But the
  solution of (\ref{variational}) is unique, thanks to Stampacchia's
  Theorem (see \cite[Theorem~2.1]{Stampacchia}). So,
  $u_{1}^{+}=u_{2}^{+}$. Therefore, (\ref{Sunique}) holds for
  $u_{1}^{-}$ and $u_{2}^{-}$, but $u_{1}^{+}=u_{2}^{+}$ shows that
  the right hand side of both problems are the same, and since
  $s(\cdot,\cdot)$ is an elliptic bilinear form, then
  $u_{1}^{-}=u_{2}^{-}$. Thus, $u_1=u_1^++u_1^- = u_2^++u_2^- = u_2$.
\end{proof}

\begin{Remark}\label{Rem-uh-}	
  We finish this section by remarking that the complementary part
  $u_h^-$ of $u_h^{}$ has a local support. In fact, we first remark that
  \begin{align*}
    (u_{h}^{+}+u_{h}^{-})^{+}=(u_{h}^{+}+u_{h}-u_{h}^{+})^{+}=u_{h}^{+}.
  \end{align*}
  This implies that $u_{h}^{-}(\boldsymbol{x}_{i})\neq 0$ if and only
  if $u_{h}^{+}(\boldsymbol{x}_{i})=\kappa$ or
  $u_{h}^{+}(\boldsymbol{x}_{i})=0$.  This shows that the support of
  $u_h^-$ is contained in the region where $u_h^+=0$ or
  $u_h^+=\kappa$. In other words, $u_h^-$ has a localised support,
  restricted to the regions where the constraint imposed in the
  definition of $V_\calP^+$ is active.
\end{Remark}

\section{Error analysis}\label{Sec:Error}

This section is devoted to the error analysis of the method
\eqref{eq19}.  Since our interest in this work is to provide a
discrete solution that respects the bounds given by the continuous
problem, error estimates will be proven for the constrained part
$u_h^+$.
\begin{theorem}\label{Th:error}
  Let $u\in H^{k+1}(\Omega)\cap H^{1}_{0}(\Omega)$ be the solution of
  \eqref{CDR} and $u_{h}\in V_{\mathcal{P}}$ be the solution of
  \eqref{eq19}. Then, there exists $C>0$ independent of $\mathcal{D}$,
  $\mu$, $\bbeta$, and $h$, such that
  \begin{equation}
    \|u-u_{h}^{+}\|_{h}	\leq  Ch^{k}\left(\|\mathcal{D} \|_{0,\infty,\Omega}^{\frac{1}{2}}+\mu^{-\frac{1}{2}}\|\bbeta\|_{0,\infty,\Omega}+h^{\frac{1}{2}}\|\bbeta\|_{0,\infty,\Omega}^{\frac{1}{2}}+h\mu^{\frac{1}{2}}\right)|u |_{k+1,\Omega}.\label{error}
  \end{equation} 
\end{theorem}
\begin{proof}
  As usual we decompose the error $ u-u_{h}^{+}$ as follows
  \begin{align}
    u-u_{h}^{+}=(u-\pi(u))+(\pi(u)-u_{h}^{+})=:\eta_{h}+e_{h}\,,
  \end{align}
  where $\pi$ is the $L^2(\Omega)$-orthogonal projection defined in
  \eqref{def:projection}.
  
  The bound for $\eta_h$ is a direct consequence of \eqref{proj1} and
  \eqref{equ35}. In fact, using the Cauchy-Schwarz and Young
  inequalities we get
  \begin{align*}
    \|\eta_{h}\|_{h}^{2}&=a_{J}(\eta_{h},\eta_{h})= (\mathcal{D}\nabla\eta_{h},\nabla\eta_{h})_{\Omega}+\mu (\eta_{h},\eta_{h})_{\Omega}+J(\eta_{h},\eta_{h})\\
    &\leq C \left(\|\mathcal{D} \|_{0,\infty,\Omega}|\eta_{h} |_{1,\Omega}^{2}+\mu \|\eta_{h} \|_{0,\Omega}^{2}+h\|\bbeta \|_{0,\infty,\Omega}|\eta_h |_{1,\Omega}^{2}\right)\\
    &\leq C h^{2k}\left(\|\mathcal{D} \|_{0,\infty,\Omega}+h\|\bbeta \|_{0,\infty,\Omega}+h^{2}\mu \right)|u |_{k+1,\Omega}^{2}.
  \end{align*}		
  
  Next, to bound $\|e_{h}\|_{h}$  we use the ellipticity of $a_J(\cdot,\cdot)$ to get
  \begin{equation}\label{eh=I+II}
    \|e_{h} \|_{h}^{2}=a_{J}^{}(\eta_{h},e_{h})-a_{J}^{}(u-u_{h}^{+},\pi(u)-u_{h}^{+})=:{\rm I}+{\rm II}.
  \end{equation}
  We start decomposing I as follows:
  \begin{equation}\label{aux-error-1}
    {\rm I}
    =
    (\mathcal{D}\nabla\eta_{h},\nabla e_{h})_{\Omega}+(\bbeta\cdot \nabla\eta_{h},e_{h})_{\Omega}+\mu (\eta_{h},e_{h})_{\Omega}+J(\eta_{h},e_{h})
    =:
    {\rm (a)+(b)+(c)+(d)}\,,
  \end{equation}	
  and bound each one of the above terms separately. Firstly, by
  Cauchy-Schwarz inequality, then using \eqref{proj1} we have
  \begin{equation}
    \label{eq:a-bound}
      {\rm (a)}
      \leq
      \| \mathcal{D}\|_{0,\infty,\Omega}^{\frac{1}{2}}|\eta_{h} |_{1,\Omega}\|\mathcal{D}^{\frac{1}{2}}\nabla e_{h} \|_{0,\Omega}
     \leq
      Ch^{k}\| \mathcal{D}\|_{0,\infty,\Omega}^{\frac{1}{2}}|u |_{k+1,\Omega}\| e_{h} \|_{h}\,.
  \end{equation}
  To bound (b) we first integrate by parts,
  use the orthogonality of $\pi$, and Lemma~\ref{lem33}.
  \begin{equation}
    \label{eq:b-bound}
    \begin{split}
      {\rm (b)}
      &=
      (\pi(u)-u,\bbeta\cdot\nabla e_{h})_{\Omega}
      \\
      &=
      \big(\pi(u)-u, \bbeta\cdot\nabla e_{h} -\pi(\bbeta\cdot\nabla e_{h})\big)_{\Omega}
      \\
      &\leq
      \|\eta_{h} \|_{0,\Omega}\| \bbeta\cdot\nabla e_{h} -\pi(\bbeta\cdot\nabla e_{h})\|_{0,\Omega}
      \\
      &\leq Ch^{k+\frac{1}{2}}\|\bbeta\|_{0,\infty,\Omega}^{\frac{1}{2}}|u |_{k+1,\Omega}\,\|e_h\|_h\,.
    \end{split}
  \end{equation}
  The term (c) is handled analagously to (a)
  \begin{equation}
    \label{eq:c-bound}
      {\rm (c)}
      \leq
      \mu^{\frac{1}{2}}\|\eta_{h} \|_{0,\Omega}\| \mu^{\frac{1}{2}}e_{h}\|_{0,\Omega}
      \leq
      C h^{k+1}\mu^{\frac{1}{2}}|u |_{k+1,\Omega}\| e_{h}\|_h\,.
  \end{equation}
  Finally, for (d), since $J(\cdot,\cdot)$ is semi-positive definite
  and symmetric we apply Cauchy-Schwarz's inequality followed by
  Lemma~\ref{lem33} and \eqref{proj1}
  \begin{equation}
    \label{eq:d-bound}
    \begin{split}
      {\rm (d)}
      &\leq
      J(\eta_{h},\eta_{h})^{\frac{1}{2}}J(e_{h},e_{h})^{\frac{1}{2}}
      \\
      &\leq Ch^{k+\frac{1}{2}}\|\bbeta\|_{0,\infty,\Omega}^{\frac{1}{2}}|u |_{k+1,\Omega}\,\|e_h\|_h\,.
    \end{split}
  \end{equation}
  Substituting \eqref{eq:a-bound}, \eqref{eq:b-bound},
  \eqref{eq:c-bound}, and \eqref{eq:d-bound} into \eqref{aux-error-1} we
  obtain the following bound for I:
  \begin{equation}\label{I1}
    \textrm{I}
    \leq
    Ch^{k}\left( \|\mathcal{D} \|_{0,\infty,\Omega}^{\frac{1}{2}}
    +
    h^{\frac{1}{2}}\|\bbeta\|_{0,\infty,\Omega}^{\frac{1}{2}}
    +
    h\mu^{\frac{1}{2}} \right)|u |_{k+1,\Omega}\|e_h\|_h\,.
  \end{equation} 
  To bound II we first recall that $(\pi(u))^{-}=\pi(u)-(\pi(u))^{+}$ and then
  \begin{align*}
    {\rm II}
    =
    a_{J}(u-u_{h}^{+},e_{h})
    =
    a_{J}(u-u_{h}^{+},(\pi(u))^{+}-u_{h}^{+})+a_{J}(u-u_{h}^{+},(\pi(u))^{-}).
  \end{align*}
  Now, thanks to the regularity of $u$, we note that
  $J(u,(\pi(u))^{+}-u_{h}^{+})=0$. Hence, using that $u_h^{+}$ solves
  \eqref{variational}, we have
  \begin{align*}
    a_{J}(u-u_{h}^{+},(\pi(u))^{+}-u_{h}^{+})&=a_{J}(u,(\pi(u))^{+}-u_{h}^{+})-a_{J}(u_{h}^{+},(\pi(u))^{+}-u_{h}^{+})\\
    &=\langle f,(\pi(u))^{+}-u_{h}^{+}\rangle_{\Omega}-a_{J}(u_{h}^{+},(\pi(u))^{+}-u_{h}^{+})\leq 0.
  \end{align*}
  Therefore, 
  \begin{align*}
    {\rm II}=a_{J}(u-u_{h}^{+},e_{h})\leq a_{J}(u-u_{h}^{+},(\pi(u))^{-}).
  \end{align*}
  The term in the right-hand side of the above inequality is, in
  essence, a consistency error, and will require special treatment.
  Let $i_{h} $ be the Lagrange interpolant defined in
  \eqref{lagrange}. Since $u(x)\in [0,\kappa]$ a.e. in $\Omega$, then
  $i_{h}(u)\in V_{\mathcal{P}}^{+}$. So, $i_{h}(u)=(i_{h}(u))^{+}\in
  V_{\mathcal{P}}^{+}$, which implies $(i_{h}(u))^{-}=0$. Therefore,
  \begin{align*}		 						
    a_{J}(u-u^{+}_{h},(\pi(u))^{-})=		a_{J}(u-u^{+}_{h},(\pi(u))^{-}-(i_{h}(u))^{-}).
  \end{align*}	 								
  So, using the definition of $a_J(\cdot,\cdot)$ we bound II as follows	
  \begin{equation}
    \begin{split}
      {\rm II}
      &\leq
      a_{J}(u-{u_h^{+}}  ,(\pi(u))^{-}-(i_{h}(u))^{-})
      \\
      &= (\mathcal{D}\nabla(u-u_{h}^{+}),\nabla ((\pi(u))^{-}-(i_{h}(u))^{-})_{\Omega} 
      +(\bbeta\cdot \nabla (u-u_{h}^{+}),(\pi(u))^{-}-(i_{h}(u))^{-})_{\Omega}
      \\
      &\quad +\mu (u-u_{h}^{+},(\pi(u))^{-}-(i_{h}(u))^{-})_{\Omega}
      +J(u-u_{h}^{+},(\pi(u))^{-}-(i_{h}(u))^{-})
      \\ 
      &={\rm (e)+(f)+(g)+(h)}.
    \end{split}
    \label{II45}
  \end{equation}
  We begin with (e), which can be bounded by Cauchy-Schwarz
  inequality
  \begin{equation}
    \label{eq:e-bound}
    \begin{split}
      {\rm (e)}
      &\leq
      \|\mathcal{D} \|_{0,\infty,\Omega}^{\frac{1}{2}}
      \|\mathcal{D}^{\frac{1}{2}}\nabla(u-u_{h}^{+}) \|_{0,\Omega}|(\pi(u))^{-}-(i_{h}(u))^{-} |_{1,\Omega}
      \\
      &\leq Ch^{-1}\|\mathcal{D} \|_{0,\infty,\Omega}^{\frac{1}{2}}
      \|\mathcal{D} ^{\frac{1}{2}}\nabla(u-u_{h}^{+}) \|_{0,\Omega}\|(\pi(u))^{-}-(i_{h}(u))^{-} \|_{0,\Omega},
    \end{split}
  \end{equation}
  using an inverse inequality. Then, as $(\cdot)^-$ is Lipschitz
  continuous by Lemma~\ref{lem32}
  \begin{equation*}
    \begin{split}
      {\rm (e)}
      &\leq
      Ch^{-1}\|\mathcal{D}
      \|_{0,\infty,\Omega}^{\frac{1}{2}}\|\mathcal{D}^{\frac{1}{2}}\nabla(u-u_{h}^{+}) \|_{0,\Omega}\|\pi(u)-i_{h}(u)\|_{0,\Omega}
      \\
      &\leq
      Ch^{k}\|\mathcal{D} \|_{0,\infty,\Omega}^{\frac{1}{2}}|u |_{k+1,\Omega}	\|u-u_h^+\|_h^{}\,,
    \end{split}
  \end{equation*}
  by the approximation properties for $i_h$ and $\pi$ given in
  \eqref{lagranges} and \eqref{proj1}. For (f) we begin by integrating
  by parts
  \begin{equation*}
    \begin{split}
      {\rm (f)}
      &=
      -\big(  u-u_{h}^{+},\bbeta\cdot\nabla((\pi(u))^{-}-(i_{h}(u))^{-})\big)_\Omega^{}
      \\
      &\leq \|\bbeta\|_{0,\infty,\Omega}\| u-u_{h}^{+} \|_{0,\Omega} |(\pi(u))^{-}-(i_{h}(u))^{-} |_{1,\Omega}^{}\,,
    \end{split}
  \end{equation*}
  using the Cauchy-Schwarz. inequality Then again, by an inverse estimate
  \begin{equation}
    \label{eq:f-bound}
    \begin{split}
      {\rm (f)}
      & \leq Ch^{-1}\|\bbeta\|_{0,\infty,\Omega} \| u-u_{h}^{+} \|_{0,\Omega} \|(\pi(u))^{-}-(i_{h}(u))^{-} \|_{0,\Omega}
      \\
      &\leq Ch^{k}\|\bbeta\|_{0,\infty,\Omega}\,\mu^{-\frac{1}{2}} |u |_{k+1,\Omega}\| u-u_{h}^{+}\|_h^{}\,,
    \end{split}
  \end{equation}
  again using Lipschitz continuiuty of $(\cdot)^-$ and the
  approximation properties for $i_h$ and $\pi$. Now (g) is controlled
  in much the same way using the Cauchy-Schwarz inequality
  \begin{equation}
    \label{eq:g-bound}
    \begin{split}
      {\rm (g)}
      &\leq
      \mu^{\frac{1}{2}}\|\mu^{\frac{1}{2}}(u-u_{h}^{+}) \|_{0,\Omega}\|(\pi(u))^{-}-(i_{h}(u))^{-} \|_{0,\Omega}
      \\
      &\leq Ch^{k+1}\mu^{\frac{1}{2}}|u |_{k+1,\Omega}
			\|u-u_{h}^{+} \|_h^{}\,,
    \end{split}
  \end{equation}
  and by Lipschitz continuiuty of $(\cdot)^-$ and the approximation
  properties for $i_h$ and $\pi$. Finally for (h) Cauchy-Schwarz implies
  \begin{equation*}
    \begin{split}
      {\rm (h)}
      &\leq
      J(u-u_{h}^{+},u-u_{h}^{+})^{\frac{1}{2}}J((\pi(u))^{-}-(i_{h}(u))^{-},(\pi(u))^{-}-(i_{h}(u))^{-})^{\frac{1}{2}}
      \\
      &\leq
      C\|\bbeta\|_{0,\infty,\Omega}^{\frac{1}{2}}\,h^{\frac{1}{2}} |(\pi(u))^{-}-(i_{h}(u))^{-} |^{}_{1,\Omega}\,
      J(u-u_{h}^{+},u-u_{h}^{+})^{\frac{1}{2}},
    \end{split}
  \end{equation*}
  using Lemma \ref{lem33}. Now by an inverse inequality
  \begin{equation}
    \label{eq:h-bound}
    \begin{split}
      {\rm (h)}
      &\leq
      C\|\bbeta\|_{0,\infty,\Omega}^{\frac{1}{2}}\, h^{-\frac{1}{2}} \|(\pi(u))^{-}-(i_{h}(u))^{-} \|^{}_{0,\Omega}\,
      J(u-u_{h}^{+},u-u_{h}^{+})^{\frac{1}{2}}
      \\
      &\leq
      Ch^{k+\frac{1}{2}}\|\bbeta\|_{0,\infty,\Omega}^{\frac{1}{2}}|u |_{k+1,\Omega}\|u-u_h^+\|_h^{},
    \end{split}
  \end{equation}
  through the approximability of $i_h$ and $\pi$ and by Lipschitz
  continuiuty of $(\cdot)^-$.
  
  Combining (\ref{eq:e-bound}), (\ref{eq:f-bound}), (\ref{eq:g-bound})
  and (\ref{eq:h-bound}) we arrive to the following bound for II:
			 \begin{equation}		 						
			 {\rm II} \leq Ch^{k}\left(\|\mathcal{D} \|_{0,\infty,\Omega}^{\frac{1}{2}}+\mu^{-\frac{1}{2}}\|\bbeta\|_{0,\infty,\Omega}+h^{\frac{1}{2}}\|\bbeta\|_{0,\infty,\Omega}^{\frac{1}{2}}+h \mu^{\frac{1}{2}}\right)\|u-u_{h}^{+} \|_{h}|u |_{k+1,\Omega}.\label{II2}
			 \end{equation}

			Hence, inserting \eqref{I1}  and \eqref{II2} into \eqref{eh=I+II}, and using Young's inequality, we obtain the following bound for $\|e_{h} \|_{h}$:
			\begin{equation*}
			\|e_{h} \|_{h}^{2}
			\leq Ch^{2k}\big( \|\mathcal{D} \|_{0,\infty,\Omega}^{\frac{1}{2}} + h^{\frac{1}{2}}\|\bbeta\|_{0,\infty,\Omega}^{\frac{1}{2}}+
			\mu^{-\frac{1}{2}}\|\bbeta\|_{0,\infty,\Omega}+
			h\mu^{\frac{1}{2}}\big)^2\,|u|_{k+1,\Omega}^2 +\frac{1}{2}\|e_{h}\|_{h}^{2}+\frac{1}{8}\|u-u_{h}^{+}\|_{h}^{2}.
			\end{equation*} 
	Finally, collecting the bounds that have been obtained for $\|e_{h} \|_{h}$	and $	\|\eta_{h} \|_{h}$ yields		
				\begin{align*}
		\|u-u_{h}^{+}\|_{h}	&\leq \|e_{h} \|_{h}+\|\eta_{h} \|_{h}\\
		&\leq {Ch^k}\, \left(\|\mathcal{D} \|_{0,\infty,\Omega}^{\frac{1}{2}}+\mu^{-\frac{1}{2}}\|\bbeta\|_{0,\infty,\Omega}+h^{\frac{1}{2}}\|\bbeta\|_{0,\infty,\Omega}^{\frac{1}{2}}+h\mu^{\frac{1}{2}}\right)|u |_{k+1,\Omega}+\frac{1}{2}\|u-u_{h}^{+}\|_{h}\,,
				\end{align*} 
				and \eqref{error} follows rearranging terms.
		\end{proof}	

{
\subsection{The extension to problems with non-homogeneous boundary conditions}
Although the presentation of the method and its analysis have been
done assuming that the Dirichlet boundary conditions are homogeneous,
many problems of practical interest involve non-homogeneous boundary
conditions. So, in this section we briefly describe the method for a
non-homogeneous Dirichlet problem. Let us assume that, instead of
\eqref{CDR} we are interested in solving
\begin{align}
   -{\rm div} (\mathcal{D}\nabla u)+\bbeta\cdot\nabla u+\mu u&= f  \hspace{1cm}{\rm in}\;\Omega, \label{CDR-NH}\\  
   u&= g  \hspace{1cm}{\rm on}\; \partial \Omega\,,\nonumber
\end{align}
where $g\in H^{\frac{1}{2}}(\partial\Omega)$, $g\ge 0$ on
$\partial\Omega$.  Thanks to the definition of $\kappa$ and the
maximum and comparison principles for partial differential equations,
$\kappa\ge \|g\|_{0,\infty,\partial\Omega}^{}$.  For simplicity, we
will assume that $g$ is the trace of a function belonging to
$\tilde{V}_\calP^{}$.  }

{Let us consider the set of nodes of the triangulation
  (including boundary nodes): $\bx_1^{},\ldots,\bx_P^{}$, and, as
  earlier in the manuscript, we denote the interior ones by
  $\bx_1^{},\ldots,\bx_N^{}, N<P$.  Let us define the following
  extension of $g$ into $\Omega$: $u_{h,g}^{}\in \tilde{V}_\calP^{}$
  defined as follows
\begin{equation}\label{extension}
u_{h,g}^{}(\bx_i^{})=\left\{ \begin{array}{cl} g(\bx_i^{}) & \textrm{if}\; i\in\{N+1,\ldots,P\}\,, \\
0 & \textrm{else}\,.\end{array}\right.
\end{equation}
With these ingredients the analogue of \eqref{eq19} for the
non-homogeneous boundary data is: Find $\tilde{u}_h^{}\in V_\calP^{}$
such that
\begin{equation}\label{FEM-NH}
a_J^{}\big((\tilde{u}_h^{}+u_{h,g}^{})^+, v_h^{}\big)+s\big((\tilde{u}_h^{}+u_{h,g}^{})^-, v_h^{}\big) = \langle f,v_h^{}\rangle_\Omega^{}
\qquad\forall\, v_h^{}\in V_\calP^{}\,.
\end{equation}
}

{The fundamental reason for choosing $u_{h,g}^{}$ as
  extension of $g$, rather than any other, is that at each node of
  $\calP$ either $\tilde{u}_h^{}$ or $u_{h,g}^{}$ are zero, and then
  the following important property holds:
\begin{equation}
(\tilde{u}_h^{}+u_{h,g}^{})^+=\tilde{u}_h^{+}+u_{h,g}^{}\,,
\end{equation}
and, as a consequence $(\tilde{u}_h^{}+u_{h,g}^{})^-=\tilde{u}_h^{-}$ (as $u_{h,g}^-=0$ since $\kappa\ge \|g\|_{0,\infty,\partial\Omega}^{}$).  Thus, 
\eqref{FEM-NH} can be rewritten as: Find $\tilde{u}_h^{}\in V_\calP^{}$
such that 
\begin{equation}\label{FEM-NH-big}
a_J^{}\big(\tilde{u}_h^{+}, v_h^{}\big)+s\big(\tilde{u}_h^{-}, v_h^{}\big) = (f,v_h^{})_\Omega^{}
- a_J^{}\big(u_{h,g}^{}, v_h^{}\big)
\qquad\forall\, v_h^{}\in V_\calP^{}\,.
\end{equation}
Now the proof of Theorem~\ref{Theorem14} remains unchanged, and thus
\eqref{FEM-NH-big} has a solution. For the uniqueness result, the same
arguments used in the proof of Lemma~\ref{lemma5} lead to the fact
that $\tilde{u}_h^+$ solves the following variational inequality:
$\tilde{u}_h^+\in V_\calP^+$, and
\begin{equation}\label{Variational-Ineq-bis}
a_J^{}(\tilde{u}_h^+,v_h^{}-\tilde{u}_h^+)\ge \langle f,v_h^{}-\tilde{u}_h^+\rangle_\Omega^{}
- a_J^{}\big(u_{h,g}^{}, v_h^{}-\tilde{u}_h^{+}\big)
\qquad\forall\, v_h^{}\in V_\calP^+\,.
\end{equation}
}

{Since \eqref{Variational-Ineq-bis} has a unique solution thanks to Stampacchia's Theorem,
the existence and uniqueness of solution for \eqref{FEM-NH} follows using exactly the same arguments
as those for  \eqref{eq19}.  Finally, to analyse the error we notice that, assuming enough regularity for the exact solution we get
\begin{equation*}
a_J^{}(u, v_h^{}-\tilde{u}_h^+)=\langle f,v_h^{}-u_h^+\rangle_{\Omega}^{}\,,
\end{equation*}
for all $v_h^{}\in V_\calP^{}$. So,  the following variational inequality holds
\begin{equation}
a_J^{}((\tilde{u}_h^{}+u_{h,g}^{})^+-u, v_h^{}-u_h^+)\ge 0\qquad\forall\, v_h^{}\in V_\calP^+\,,
\end{equation}
which is instrumental in the bound for the bound of \textrm{II} in the proof of Theorem~\ref{Th:error}.
Hence, the error analysis follows very similar arguments as those presented in the proof of Theorem~\ref{Th:error}.
}

\input{numerics}

\clearpage
\section{Conclusions and outlook} 	
\label{sec:conc}

The method proposed in this work constitutes an inexpensive and simple
way to impose global bounds in the numerical solution to
convection-diffusion equations (at least at the nodes of the
triangulation, and globally for the lowest order case $k=1$).  Other
than its well-posedness, we proved optimal order error estimates for
the constrained part $u_h^+$.  It is worth stressing that
optimal-order error estimates are not common for this type of method
(even in the piecewise linear case, see \cite{BJK23} for more
details), so we believe this result constitutes a highlight of this
paper.  The numerical experiments presented show that it constitutes a
competitive alternative to previously existing alternatives.

Looking forward, the complementary part $u_h^-$ has a localised
support. This has been checked numerically, but the important question
about whether this fact can be exploited in a posteriori error
analysis remains an open question. In addition, the generality of this
approach makes the applications of the method to unsteady nonlinear
fluids problems, for example of porous media and Allen-Cahn type,
quite natural. These, amongst other areas, are topics of ongoing
research that shall be reported in upcoming publications.

\paragraph{Acknowledgements}
The work of AA, GRB, and TP has been partially supported by the
Leverhulme Trust Research Project Grant No. RPG-2021-238.  TP is also
partially supported by EPRSC grants
\href{https://gow.epsrc.ukri.org/NGBOViewGrant.aspx?GrantRef=EP/W026899/2}
     {EP/W026899/2},
\href{https://gow.epsrc.ukri.org/NGBOViewGrant.aspx?GrantRef=EP/X017206/1}
     {EP/X017206/1}
and
\href{https://gow.epsrc.ukri.org/NGBOViewGrant.aspx?GrantRef=EP/X030067/1}
     {EP/X030067/1}. The authors also want to thank Emmanuil Geourgoulis
and Andreas Veeser for numerous very helpful and stimulating discussions.
		
\bibliography{refs.bib}
	
%
%
	
\end{document}

%% file: numerics.tex
\section{Numerical experiments}	
\label{sec:numerics}

In this section we present three series of numerical results testing
the performance of the finite element method \eqref{eq19}.  In all
numerical experiments in this section $\Omega=(0,1)^2$, and we have used
the value $\alpha=1$ in the stabilising bilinear form
$s(\cdot,\cdot)$.  We have selected three different types of meshes, the coarsest level of them are depicted
in Figure~\ref{meshs}.  While the family depicted in Figure~\ref{e11}
and \ref{c11} are symmetric and Delaunay, the mesh depicted in
Figure\ref{d11} is a non-Delaunay one, and the one depicted in \ref{Q11} is
quadrilateral. The mesh represented by \ref{d11} is generated by
displacing some interior nodes of the mesh in \ref{c11} to the right, thus
creating obtuse angles. The reason for this choice is motivated by the
fact that the discrete maximum principle fails to hold for most finite
element methods in such meshes (see, e.g., \cite{BJK23}).  In
particular, the initial datum $u_h^0$ defined below will not, in
general, belong to $V_\calP^+$.
\begin{figure}[h!]
	\centering
	\subfloat[A symmetric, Delaunay mesh. ]{\label{e11}
		\includegraphics[width=0.23\textwidth]{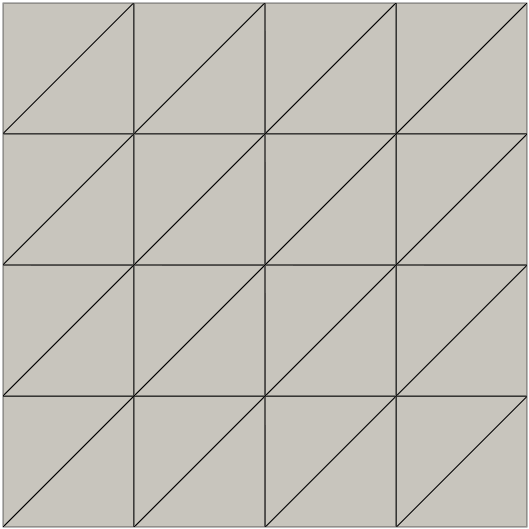}
	}
	\subfloat[A symmetric, Delaunay mesh.]{\label{c11}
		\includegraphics[width=0.23\textwidth]{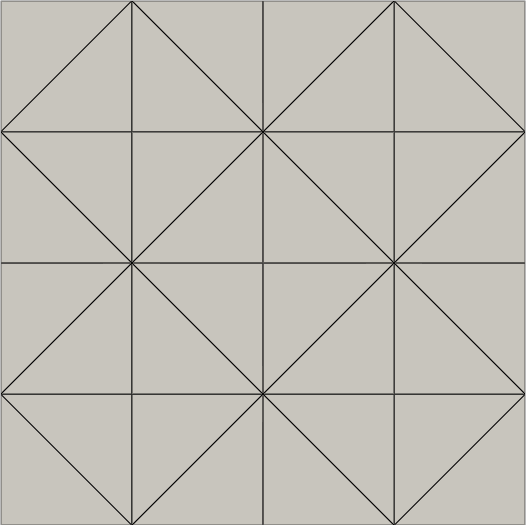}
	}
	\subfloat[A non-symmetric, non-Delaunay mesh.]{\label{d11}
		\includegraphics[width=0.23\textwidth]{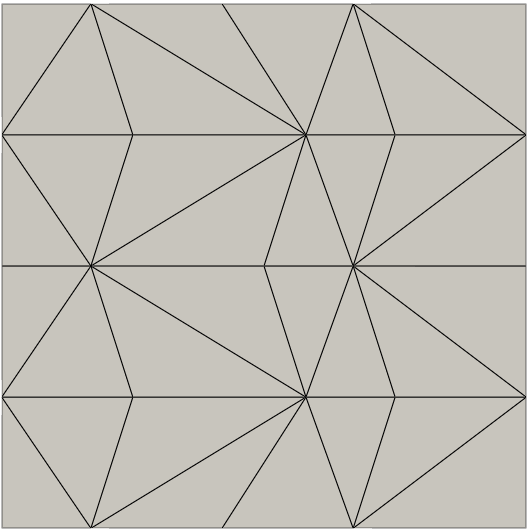}
	}
	\subfloat[A simple quadrilateral mesh.]{\label{Q11}
		\includegraphics[width=0.23\textwidth]{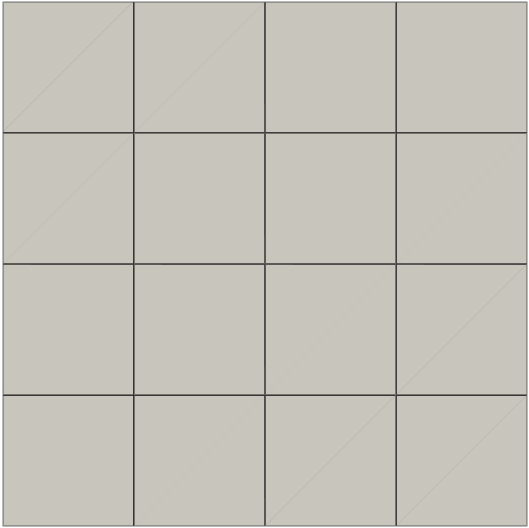}
	}
	\caption{Three coarse level indicative meshes used in the
		experiments all with $N=5$.}\label{meshs}
\end{figure}

To solve the nonlinear system associated to \eqref{FEMethod} we use
the following Richardson-like iterative method: Set $u_{h}^{0}\in
V_{\mathcal{P}}$ such that
\begin{eqnarray}
  a_J(u^{0}_{h},v_{h})= \langle f,v_{h}\rangle_{\Omega}\qquad\forall v_{h}\in V_{\mathcal{P}}\,.\label{iter}
\end{eqnarray}
Then, for $n=0,1,2,\ldots,$ find $u^{n+1}_{h}\in V_{\mathcal{P}}$ such that:
\begin{eqnarray}
  a_{J}(u^{n+1}_{h},v_{h})=a_{J}(u^{n}_{h},v_{h})+\omega \left(\langle f,v_{h}\rangle_{\Omega} - a_{J}((u^{n}_{h})^{+},v_{h})-s((u^{n}_{h})^{-},v_{h})\right)\qquad\forall v_{h}\in V_{\mathcal{P}},\label{iter1}
\end{eqnarray}
where $\omega\in(0,1]$ is a damping parameter. The iterations are
  terminated when
\begin{equation}
  \|u_{h}^{n+1}-u^{n}_{h} \|_{0,\Omega}\leq 10^{-8}. \label{stopping}
\end{equation}

In all figures, $N-1$ represents the number of divisions  in the $x$ and
$y$ directions, so the total number of vertices (including the boundary)
is $N^2$.  We test the performance of the method
asymptotically in $N$, where we use EOC as the estimated order of
convergence, and also examine the convergence of the iterative method.
We have used $\mathbb{P}_1^{},\mathbb{P}_2^{}$, and $\mathbb{P}_3^{}$ elements in the
triangular meshes, and $\mathbb{Q}_1^{}$ and $\mathbb{Q}_2^{}$ elements in the quadrilateral mesh.
		
\begin{example}[Convergence of a problem with smooth solution]\label{Example1}
  We consider $\mu=1$, $\mathcal{D}=\epsilon\begin{bmatrix} 100 &\cos(x) \\ \cos(x) & 1  \end{bmatrix}$ ($\epsilon=10^{-5}$), $\boldsymbol{\beta}=(2,1)$,
  and set $f$ such that the function $u(x,y)=100\sin(\pi x)\sin(\pi y)$
  is the analytical solution of (\ref{CDR}).  Notice that $u(x)\in
  [0,100]$, and thus we choose $\kappa=100$.  The CIP stabilisation
  parameter is $\gamma=0.025$ in the penalty term (\ref{eq10}), and we
  have used $\omega=1$ in the iterative method \eqref{iter1}.
\end{example}

In Tables \ref{Tab11}-\ref{Tab16} we report the convergence results in both the
$\|\cdot\|_{0,\Omega}$ and $\|\cdot\|_h$ norms for $u-u_h^+$, and the
$\|\cdot\|_s$-norm for $u_h^-$, as well as the number of iterations
needed to reach convergence for the nonlinear system. The results show
an optimal order of convergence for the constrained part $u_h^+$, thus
confirming the results from Section~\ref{Sec:Error}. In addition, they
show a higher order of convergence (to zero) for the complementary
part $u_h^-$.
\begin{table}[H]
	\centering
	\begin{tabular}{|r|r|c|c|c|c|c|c|}
		\hline
		$N$ &Itr  &$\|u-u^{+}_{h} \|_{0,\Omega}$& EOC  &$\|u-u^{+}_{h} \|_{h}$  & EOC     & $	\| u^{-}_{h}\|_{s}$ & EOC \\  
		\hline\hline
		5  & 2 & 8.57e+0 &-- & 4.55e+1  &-- & 0  &--\\
		9  & 7 &  2.12e+0& 2.37 &1.95e+1  & 1.44 & 3.05e-1 &--\\ 
		17 &   6&  5.05e-1 & 2.26 & 7.71e+0 & 1.46 & 2.74e-1 &0.17\\ 
		33  & 7 &1.23e-1 & 2.12 & 2.89e+0  & 1.47 &      6.64e-2 & 2.13\\  
		65 &  7 & 3.09e-2 & 2.03 &  1.06e+0 & 1.47&1.27e-2 & 2.44\\ 
		129 & 6 & 7.80e-3  & 2.00 &  3.82e-1 & 1.49 &2.29e-3 & 2.50\\ 
		\hline
	\end{tabular}
	\caption{Numerical results for Example~\ref{Example1} using $\mathbb{P}_{1}$ elements and 
		Mesh \ref{d11}.\label{Tab11}}
\end{table}
\begin{table}[H]
	\centering
	\begin{tabular}{|r|r|c|c|c|c|c|c|}
		\hline
		$N$ &Itr  &$\|u-u^{+}_{h} \|_{0,\Omega}$& EOC  &$\|u-u^{+}_{h} \|_{h}$  & EOC     & $	\| u^{-}_{h}\|_{s}$ & EOC \\  
		\hline\hline
		5  & 15 & 5.51e+0 &-- & 2.73e+1 &-- & 4.43e+0  &--\\
		9  & 15 &  8.03e-1 & 3.27 & 9.79e+0  & 1.74 & 8.43e-1  &2.82\\ 
		17 &  13 &  1.38e-1 &2.76 &  3.47e+0 & 1.63 & 1.67e-1&2.54\\ 
		33  & 12 &  2.86e-2 &2.37 & 1.23e+0  & 1.56 &   3.12e-2 & 2.52\\  
		65 &  10 & 6.62e-3 & 2.15 & 4.37e-1 & 1.53 &5.70e-3 & 2.50\\ 
		129 &  9 & 1.61e-3  & 2.06 &  1.56e-1 & 1.50 &1.02e-3& 2.51\\ 
		\hline
	\end{tabular}
	\caption{Numerical results for Example~\ref{Example1} using $\mathbb{Q}_{1}$ elements and Mesh \ref{Q11}.\label{Tab13}}
\end{table}
\begin{table}[H]
	\centering
	\begin{tabular}{|r|r|c|c|c|c|c|c|}
		\hline
		$N$ &Itr  &$\|u-u^{+}_{h} \|_{0,\Omega}$& EOC  &$\|u-u^{+}_{h} \|_{h}$  & EOC     & $	\| u^{-}_{h}\|_{s}$ & EOC \\  
		\hline\hline
		5  & 15 & 2.51e+0 &-- & 8.14e+0 &-- & 1.37e+0  &--\\
		9  & 2 &  3.02e-1& 3.60 & 1.52e+0  & 2.85 & 0 &--\\ 
		17 &   2&  3.72e-2 & 3.29 & 3.03e-1 & 2.53 & 0 &--\\ 
		33  & 2 & 4.44e-3 & 3.20 & 5.78e-2  & 2.50 &      0 & --\\  
		65 &  2 & 5.37e-4 & 3.11 &  1.07e-2 & 2.48 &0 &--\\ 
		129 &  2 & 6.57e-5  & 3.03 &  1.97e-3 & 2.44 &0 & --\\ 
		\hline
	\end{tabular}
	\caption{Numerical results for Example~\ref{Example1} using $\mathbb{P}_{2}$ elements and  Mesh \ref{d11}.\label{Tab14}}
\end{table}
\begin{table}[H]
	\centering
	\begin{tabular}{|r|r|c|c|c|c|c|c|}
		\hline
		$N$ &Itr  &$\|u-u^{+}_{h} \|_{0,\Omega}$& EOC  &$\|u-u^{+}_{h} \|_{h}$  & EOC     & $	\| u^{-}_{h}\|_{s}$ & EOC \\  
		\hline\hline
		5  & 2 & 3.77e-1 &-- & 6.22e-1 &-- & 0  &--\\
		9  & 58 &  4.26e-2 & 3.70 & 9.79e-2  &3.14 & 1.85e-2  &4.06\\ 
		17 &  44 &  5.18e-3& 3.31 &  1.71e-2& 2.74 & 2.44e-3 &3.18\\ 
		33  & 28 &6.36e-4 & 3.16 &3.21e-3  & 2.52 &      2.45e-4 & 3.46\\  
		65 &  2 & 7.75e-5 &3.10 & 6.43e-4 & 2.37&5.28e-6 & 5.66\\ 
		129 &  2 & 9.20e-6  & 3.10 &  1.37e-4 & 2.25 &4.35e-7 & 3.62\\
		\hline
	\end{tabular}
	\caption{Numerical results for Example~\ref{Example1} using $\mathbb{Q}_{2}$ elements and  Mesh \ref{Q11}.\label{Tab15}}
\end{table}
\begin{table}[H]
	\centering
	\begin{tabular}{|r|r|c|c|c|c|c|c|}
		\hline
		$N$ &Itr  &$\|u-u^{+}_{h} \|_{0,\Omega}$& EOC  &$\|u-u^{+}_{h} \|_{h}$  & EOC     & $	\| u^{-}_{h}\|_{s}$ & EOC \\  
		\hline\hline
		5  & 2 & 2.85e-1 &-- & 8.75e-1  &-- & 0  &--\\
		9  &  137 &  2.69e-2& 4.01 &  9.75e-2 & 3.73 & 6.62e-3 &--\\ 
		17 &   71 &  2.40e-3 &4.16 &  1.18e-2 &3.32& 3.75e-4 &4.51\\ 
		33  & 2 & 1.70e-4 &3.99 &1.25e-3  & 3.38 &      8.77e-7 & 9.13\\  
		65  & 2 & 9.66e-6& 4.23 &1.17e-4 & 3.49 &  0 & --\\ 
		129 &  2 & 5.03e-7 & 4.26 &  1.02e-5 &3.51&   0 & --\\ 
		\hline
	\end{tabular}
	\caption{Numerical results for Example~\ref{Example1} using $\mathbb{P}_{3}$ elements and  Mesh \ref{d11}.\label{Tab16}}
\end{table}

\begin{example}[A problem with two inner layers]\label{Example2}
  For this test case and the following one 
the diffusion in \eqref{CDR} is given by
$\mathcal{D}=\epsilon\mathcal{I}$, where $\epsilon>0$.
  We now approximate the solution of \eqref{CDR} for
  $f=0$, $\mu=0$, $\epsilon=10^{-5}$, and
  $\boldsymbol{\beta}=(-y,x)$. Homogeneous Neumann boundary
  conditions are imposed on exit (that is, at the lines $x=0$ and $y=1$), while the following
  (discontinuous) Dirichlet datum is imposed on entry ($x=1$ and $y=0$):
  \begin{equation}
    g(x,y)
    =
    \left\{
    \begin{split}
      &0 \quad \text{ if } x\leq \tfrac{1}{3} \text{ and } y=0
      \\
      &\tfrac{1}{2}  \quad \text{ if } x\in \left(\tfrac{1}{3}, \tfrac{2}{3}\right) \text{ and } y=0,
      \\
      &1  \quad \text{ otherwise.}
    \end{split}
    \right. 
  \end{equation}
\end{example}
The goal of this numerical experiment is twofold. First, we aim at
testing the capabilities of the current bound preserving method (BPM)
\eqref{eq19} to suppress over- and under-shoots in regions where the
contraint is not needed.  In fact, for this example $\kappa=1$ in the
whole domain, while there are two inner layers inside the domain where
the solution varies rapidly from $0$ to approximately $0.5$, and a
second one where it does from $0.5$ to $1$. Now, around those layers
the BPM will control the undershoot at $u_h^+=0$, but there is no
explicit control on the region where the solution is approximately
$0.5$. Our intent is to assess the capability of the current method to
suppress the possible overshoot at that region, even if the nonlinear
stabilisation is not active in it.  The second goal of this numerical
experiment is to provide numerical evidence that the addition of the
(linear) CIP stabilisation on $u_h^+$ has a positive impact on the
performance of the method.
				
We start addressing the latter of our objectives.  In
Tables~\ref{tabEx2-1} and \ref{tabEx2-2} we report the number of iterations required by
the nonlinear solver. In all our simulations we have set a maximum
number of iterations to $3,000$, and whenever we have reached this
value, the solver stops and we report ``NC'' which represents
non-convergence.  As can be seen in Tables~\ref{tabEx2-1} and \ref{tabEx2-2}, if the
linear CIP stabilisation is not added to the formulation then the
nonlinear iteration process is more prone to non-convergence, thus
giving one more argument in favour of adding linear stabilisation to
\eqref{eq19}. For this particular example we have chosen to use the
stabilising term \eqref{eq121} with stabilising parameter
$\gamma_{\bbeta} = 0.05$, as it is the one that has provided the best
numerical results in terms of sharpness of the interior layers.

We have run the experiments using meshes from Figure \ref{e11} and
Figure \ref{d11} for a variety of values for $N$.  For these
experiments, for $\mathbb{P}_{1}$, $\mathbb{Q}_{1}$ and $\mathbb{Q}_{2}$ we have used $\omega=0.1$ in \eqref{iter1} for the BPM
and $\omega=0.05$ if CIP is removed, that is, choosing
$\gamma_\bbeta=0$ (it  is worth mentioning that making $\omega$ smaller makes the
convergence of the iterative solver more likely).
\begin{table}[H]
	\centering
	\begin{tabular}{|c|l|c|c|c|c|c|c|}
		\hline
		& $N$	&	5  & 9 & 17 & 33 & 65 & 129 \\\hline
		\multirow{ 2}{*}{\textrm{Mesh \ref{e11}}}& $\gamma_\beta=0.05$	&	82  & 96 & 122  & 124& 113 & 98 \\ 
		& $\gamma_\beta=0$		&	228  & 1702 & NC & NC & NC & NC \\  \hline
		\multirow{ 2}{*}{\textrm{Mesh \ref{d11}}}&	$\gamma_\beta=0.05$	&	140 &148 & 174  & 137 & 123 &  111 \\ 
		& $\gamma_\beta=0$	         & 126 & NC & NC  & NC & NC & NC  \\ \hline
	\end{tabular}
	\caption{Number of iterations for the fixed point linearisation
		(\ref{iter1}) needed to reach convergence using $\mathbb{P}_{1}$ elements and the meshes given in
		Figures \ref{e11} and \ref{d11}.}\label{tabEx2-1}
\end{table}		

\begin{table}[H]
	\centering
	\begin{tabular}{|c|l|c|c|c|c|c|c|}
		\hline
		& $N$	&	5  & 9 & 17 & 33 & 65 & 129 \\\hline
		\multirow{ 2}{*}{$\mathbb{Q}_{1}$}& $\gamma_\beta=0.05$	&	72  & 128 & 136  & 151& 159 & 190 \\ 
		& $\gamma_\beta=0$		&	1901  & NC & NC & NC & NC & NC \\  \hline
		\multirow{ 2}{*}{$\mathbb{Q}_{2}$}&	$\gamma_\beta=0.05$	&	283 & 243 & 360 & 315& 339 & 258 \\ 
		& $\gamma_\beta=0$	         &	2034 & NC & NC  & NC & NC & NC  \\ \hline
	\end{tabular}
	\caption{Number of iterations for the fixed point linearisation
		(\ref{iter1}) needed to reach convergence using $\mathbb{Q}_{1}$ and $\mathbb{Q}_{2}$ elements and the mesh given in
		Figure \ref{Q11}.}\label{tabEx2-2}
\end{table}		

We next address the sharpness in the approximation of the interior
layers. In Figure \ref{ex6fig6}-\ref{ex6fig6Q1} we depict the approximate solution
for the meshes from Figure \ref{e11} and \ref{d11} and observe the
lack of significant oscillations in the vicinity of the layers, even
for mesh \ref{d11}, which is non-Delaunay.

For comparison purposes, we have also approximated the same problem
using the (linear) CIP method (with the stabilising term \eqref{eq121}
and $\gamma_{\bbeta}=0.05$), and the Algebraic Flux Correction (AFC)
scheme, as written in \cite{BBK17-NumMath} (the latter only for the
case of $\mathbb{P}_1^{}$ elements).  For this last method it is known
that it respects the discrete maximum principle (at least in Delaunay
meshes in two space dimensions), and thus the results are expected to
lie within the bounds, at least for the mesh from Figure~\ref{e11}. In
our experiments we have used the values $p=8$ and $\gamma_0=0.75$,
(see \cite{BBK17-NumMath} for details on the formulation of the
method).  We have carried out the comparison by taking a cross-section
along the line $y=x$. We focus our attention in two main points,
namely, suppression of over and undershoots in the numerical solution,
and in how diffused the interior layers are.  We depict zooms of the
cross-sections at the onset of the layers. We observe that, as
expected, the CIP method by itself presents over- and under-shoots,
while the BPM and AFC method do not. In fact, the function $u_h^+$
presents significantly smaller oscillations than CIP, while showing an
accuracy comparable to that of the AFC method regarding the sharpness
of the layers.

The same comparison has been carried out using the non-Delaunay mesh
from Figure \ref{d11}, and very similar conclusions are drawn.  
Interestingly,  when employing the BPM method for $\mathbb{P}_{2}^{}$, $\mathbb{Q}_{2}^{}$ and higher-order elements, even if the bounds are only imposed
at the nodes, no noticeable undershoots (in terms of violations of the physical bounds) have been observed in the numerical solution.

  \begin{figure}[H]
  	\centering
  	\subfloat[Approximation using mesh \ref{e11}]{
  		\includegraphics[width=0.4\textwidth]{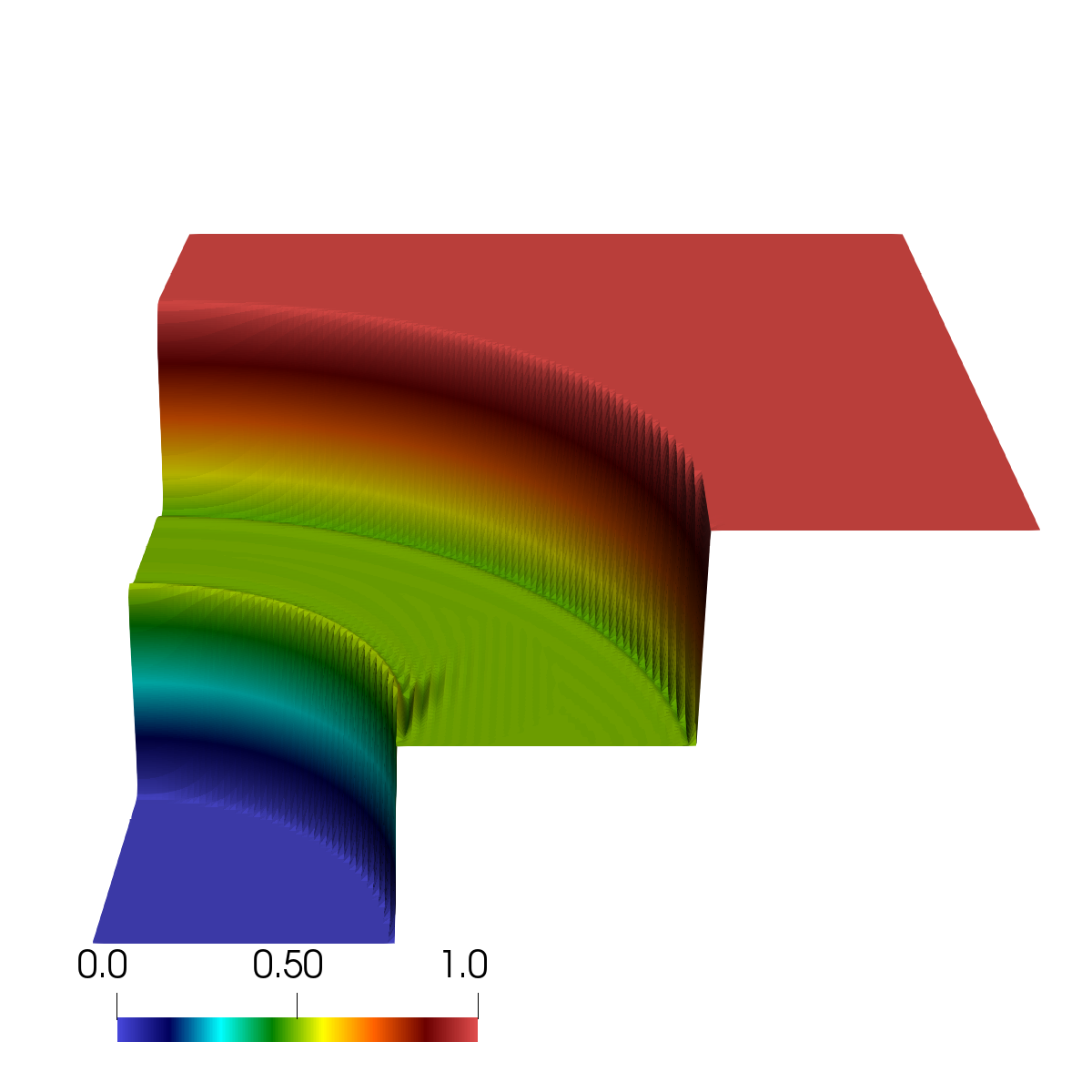}
  	}
  	\subfloat[Approximation using mesh \ref{d11}]{
  		\includegraphics[width=0.4\textwidth]{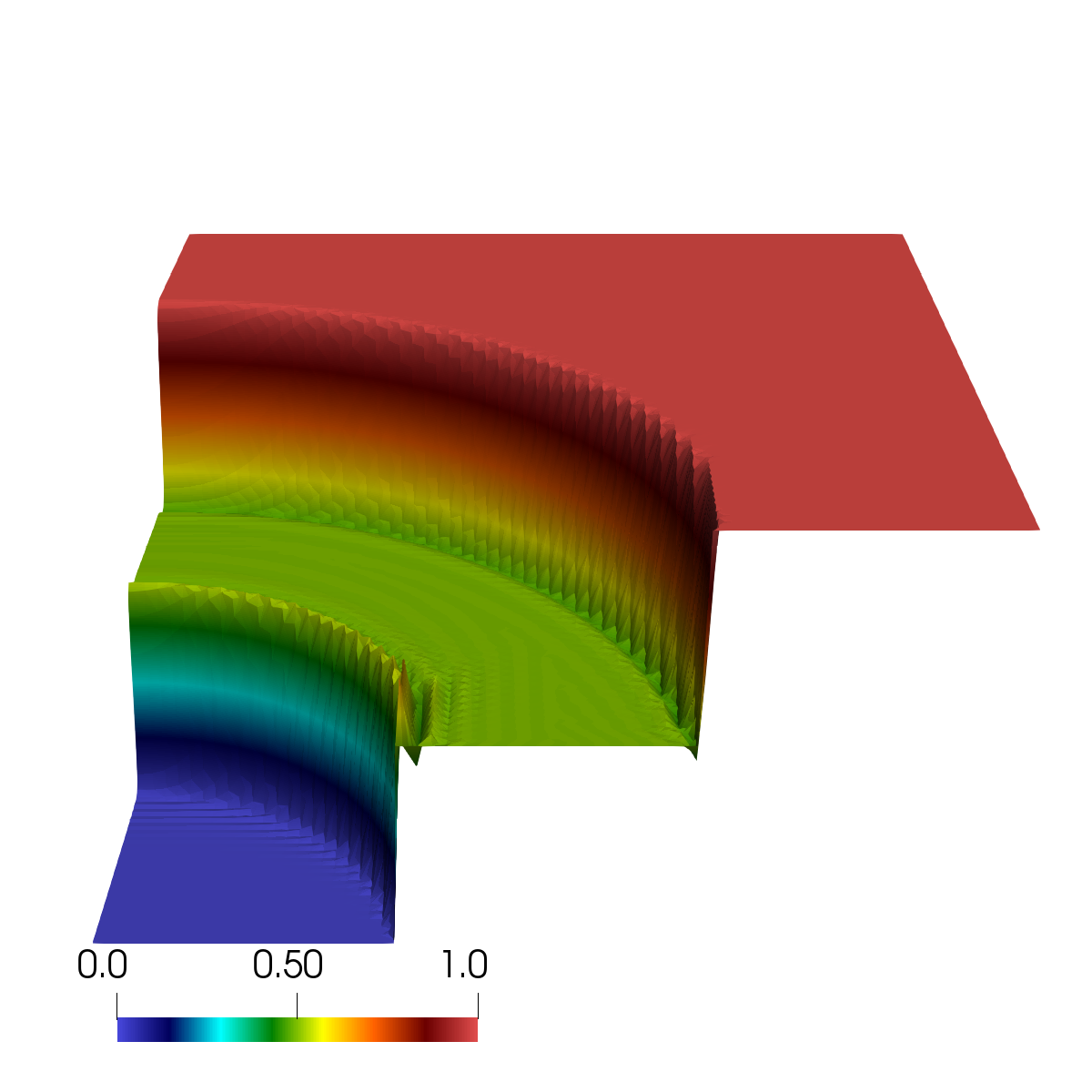}
  	}
  	\vspace{-10pt}
  	\\
  	\subfloat[Cross-section using mesh \ref{e11}]{
  		\includegraphics[width=0.4\textwidth]{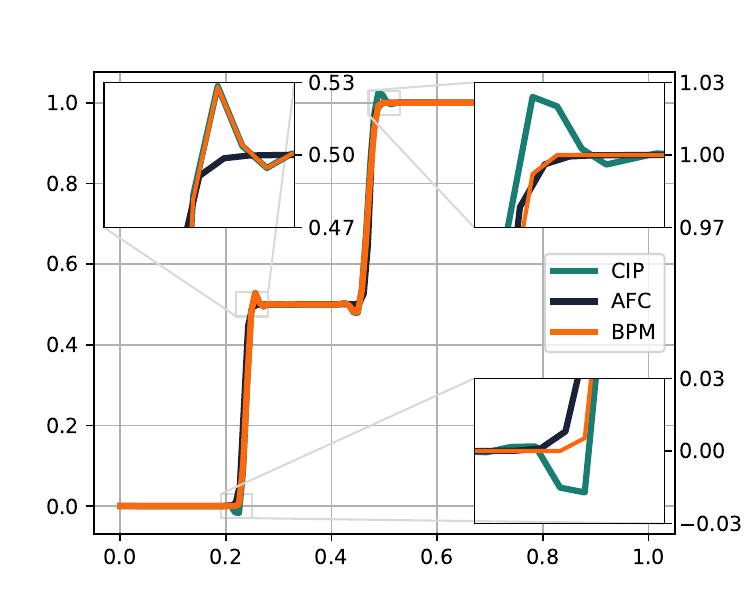}
  	}
  	\subfloat[Cross-section using mesh \ref{d11}]{
  		\includegraphics[width=0.4\textwidth]{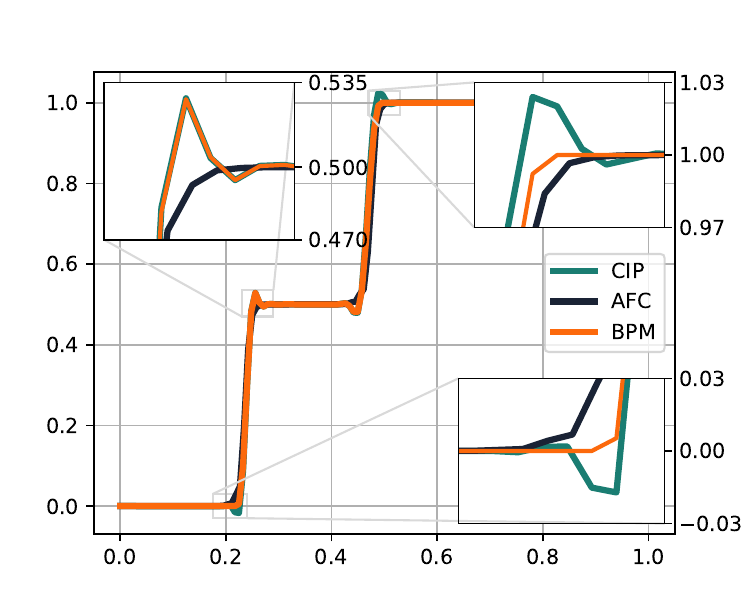}
  	}
  	\caption{
          The approximation of the solution of Example
  		\ref{Example2} by the bound preserving method (BPM), using $\mathbb{P}_{1}$ elements and the
  		meshes given in Figures \ref{e11} and \ref{d11} with
  		$N=129$. Cross-sections taken about $y=x$ plane of the solution
  		of the BPM, CIP and AFC. For AFC $p=8$ and for BPM and CIP the
  		penalty (\ref{eq121}) $\gamma_{\bbeta}=0.05$ and $\omega=0.1$ has been used. 
  		\label{ex6fig6}
  	}
  \end{figure}
 
  \begin{figure}[H]
  	\centering
  	\subfloat[Approximation using $\mathbb{P}_{2}$ elements and mesh \ref{d11}]{
  		\includegraphics[width=0.37\textwidth]{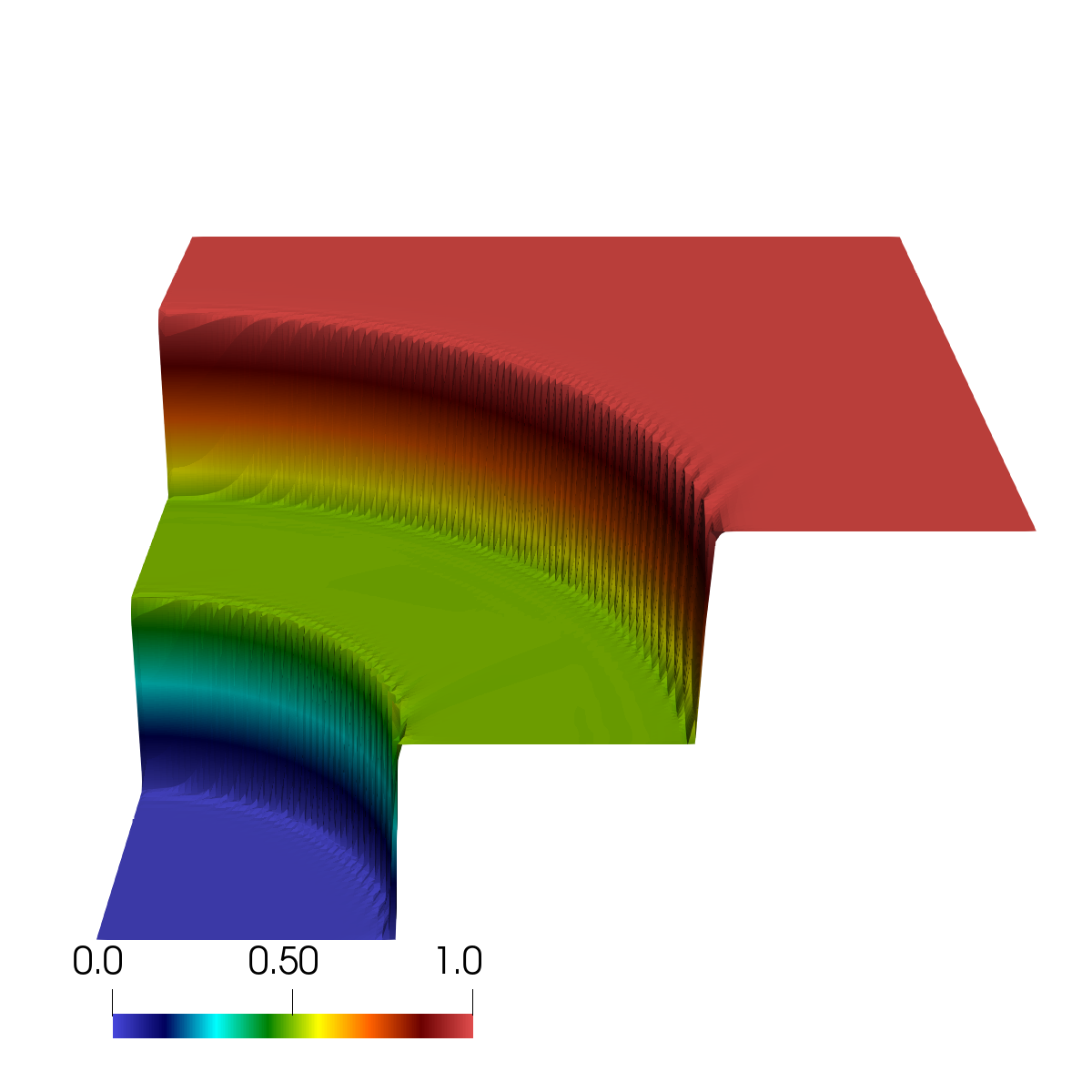}
  	}
  	\subfloat[Approximation using $\mathbb{P}_{3}$ elements and mesh \ref{d11}]{
  		\includegraphics[width=0.37\textwidth]{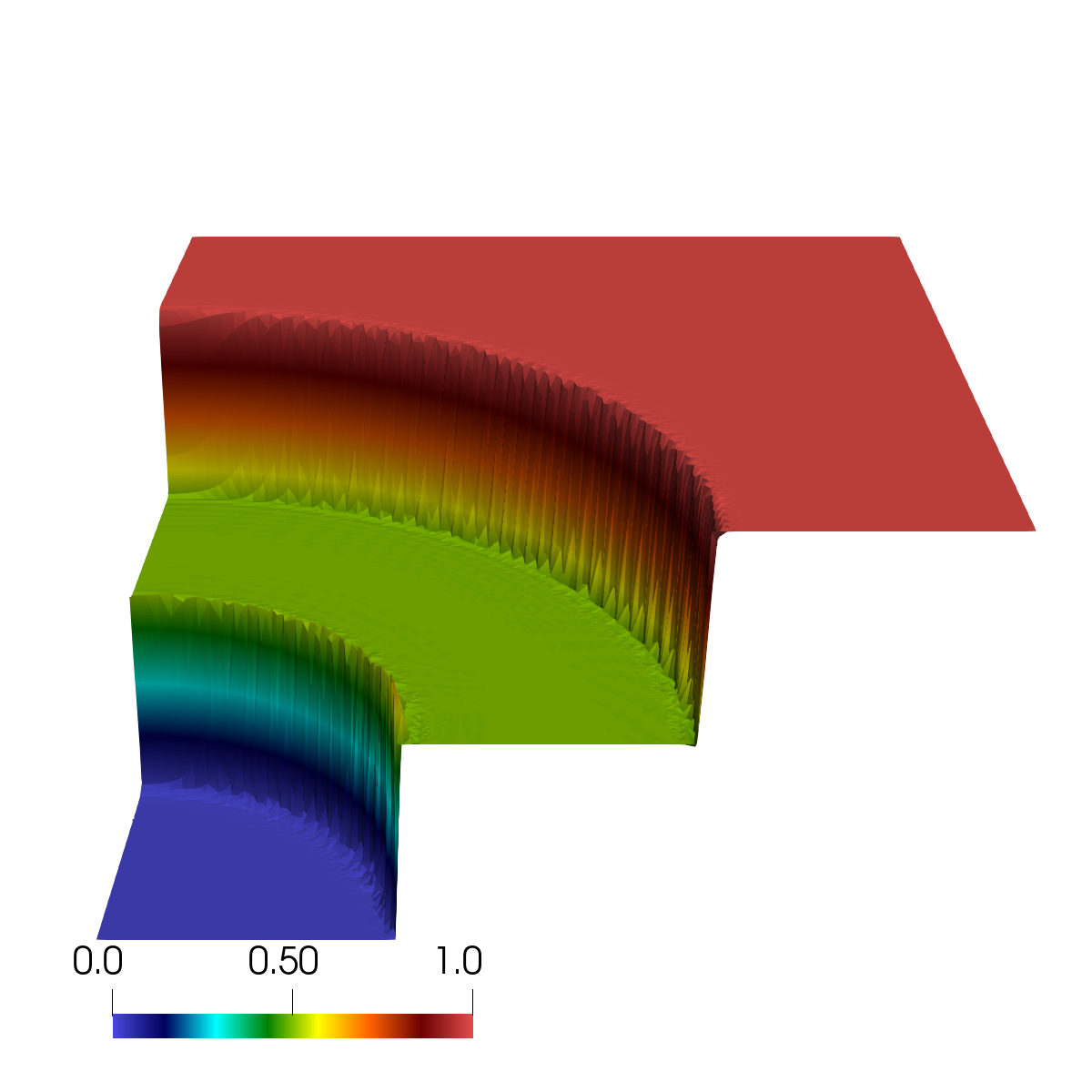}
  	}
  	\vspace{-10pt}
  	\\
  	\subfloat[Cross-section using $\mathbb{P}_{2}$ elements and mesh \ref{d11}]{
  		\includegraphics[width=0.4\textwidth]{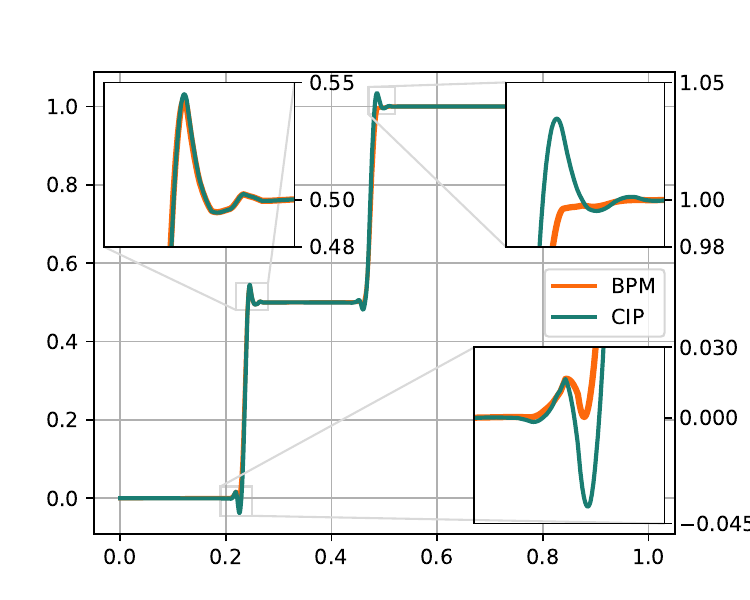}
  	}
  	\subfloat[Cross-section using $\mathbb{P}_{3}$ elements and mesh \ref{d11}]{
  		\includegraphics[width=0.4\textwidth]{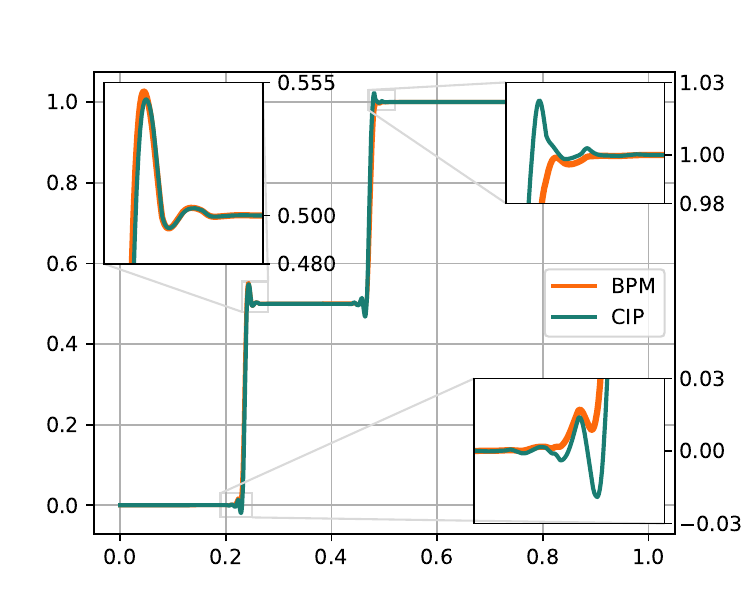}
  	}
  	\caption{
          The approximation of the solution of Example  \ref{Example2} by the bound preserving
          method (BPM), using $\mathbb{P}_{2}$ and $\mathbb{P}_{3}$ elements and the meshes given in Figure
 \ref{d11} with $N=129$. Cross-sections taken along the line $y=x$. For both methods
          the penalty \eqref{eq121} with $\gamma_{\bbeta}=0.05$ was used
          ($\omega=0.05$). For plotting these cross-sections, 10,000 equidistant points were chosen along the line $y=x$, and the values of the approximated solution have been plotted at these points.
  		\label{ex6fig6p2}
  	}
  \end{figure}

  \begin{figure}[H]
  	\centering
  	\subfloat[Approximation using $\mathbb{Q}_{1}$ elements.]{
  		\includegraphics[width=0.37\textwidth]{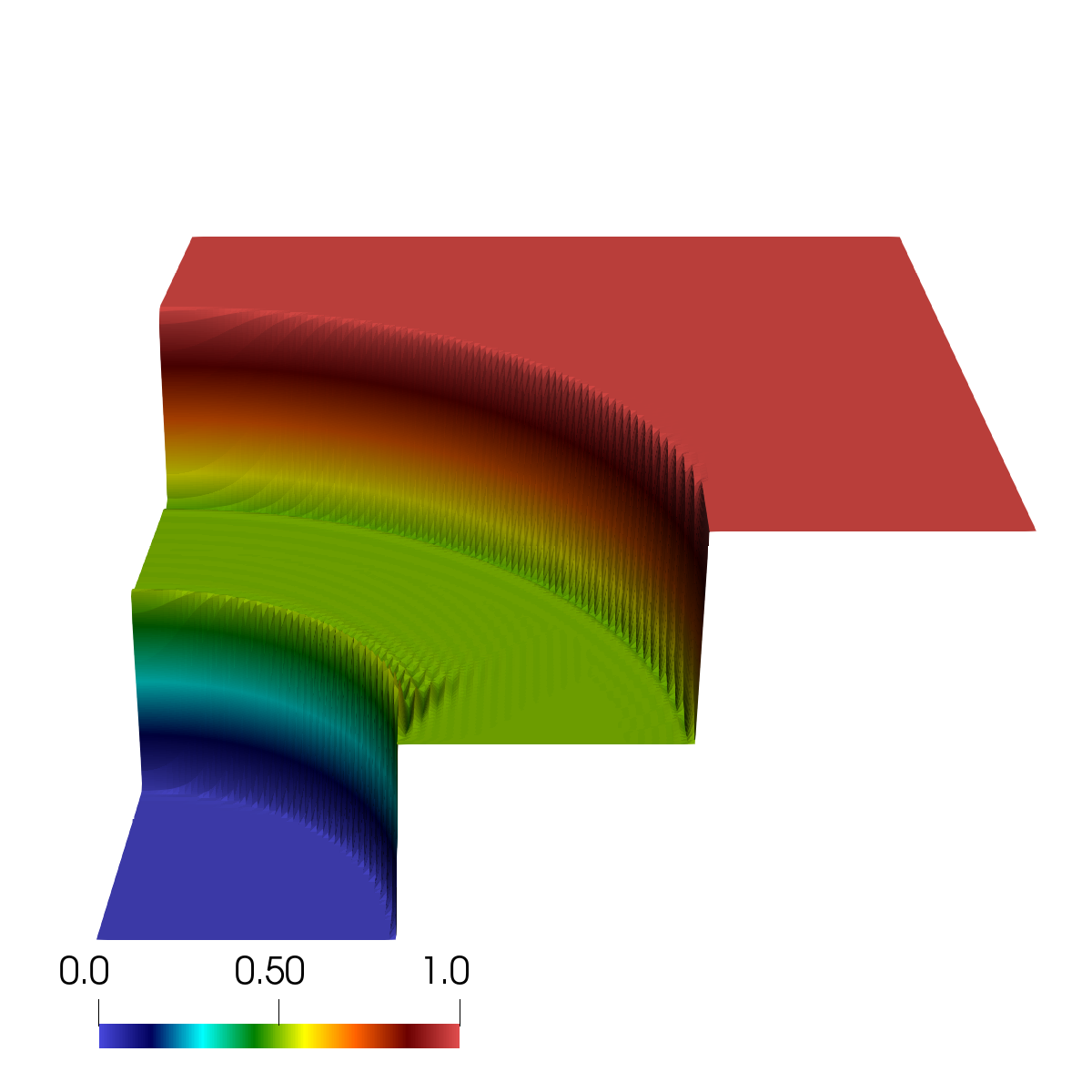}
  	}
  	\subfloat[Approximation using $\mathbb{Q}_{2}$ elements.]{
  		\includegraphics[width=0.37\textwidth]{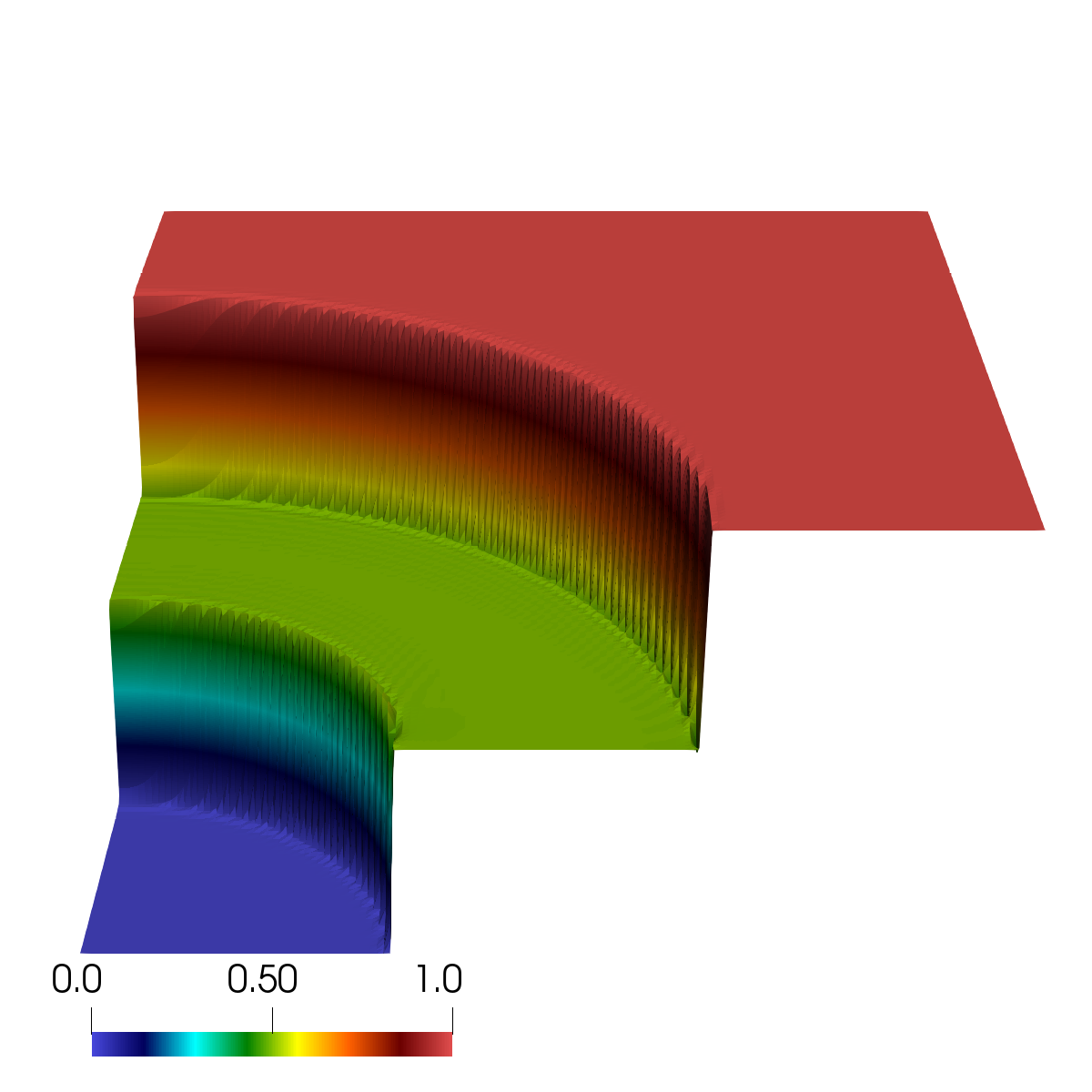}
  	}
  	\vspace{-10pt}
  	\\
  	\subfloat[Cross-section taken using $\mathbb{Q}_{1}$ elements.]{
  		\includegraphics[width=0.4\textwidth]{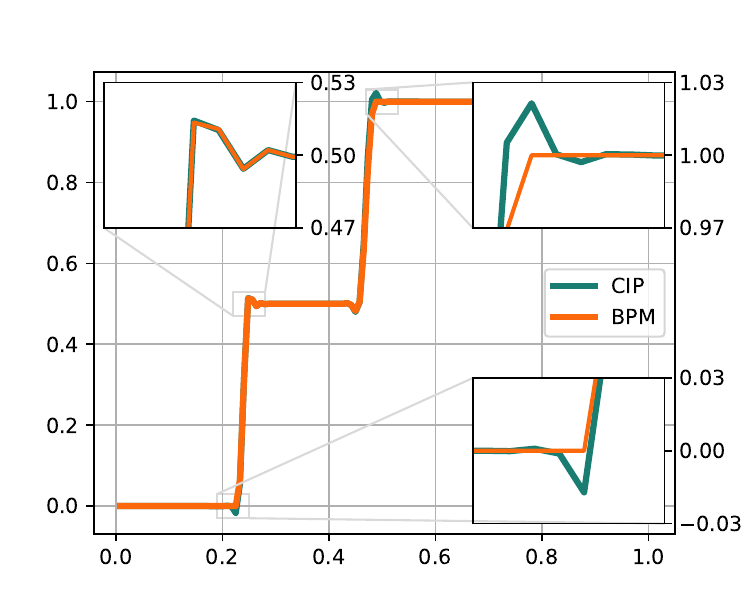}
  	}
  	\subfloat[Cross-section taken using $\mathbb{Q}_{2}$ elements.]{
  		\includegraphics[width=0.4\textwidth]{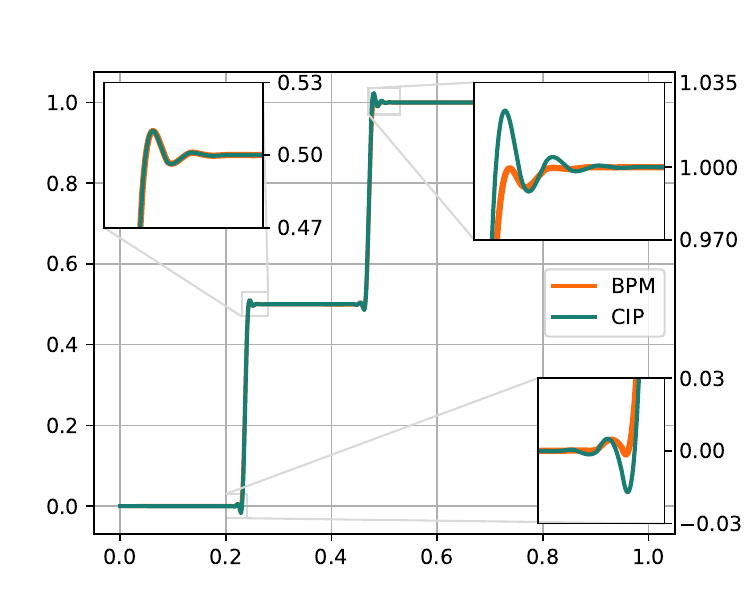}
  	}
  	\caption{ The approximation of the
          solution of Example \ref{Example2} by the bound preserving method (BPM), using $\mathbb{Q}_{2}$ and
          $\mathbb{Q}_{1}$ elements and
          the mesh given in Figure \ref{Q11} with
          $N=129$. Cross-sections taken about $y=x$ plane of the
          solution of the BPM and CIP. For BPM and CIP the penalty
          (\ref{eq121}) $\gamma_{\bbeta}=0.05$ was used
          ($\omega=0.1$). For plotting the cross-sections with $\mathbb{Q}_{2}$ elements, 10,000 equidistant points were chosen along the line $y=x$, and the values of the approximated solution have been plotted at these points.
  		\label{ex6fig6Q1}
  	}
  \end{figure}
  
  \begin{example}[A problem with an inner and a boundary layer]\label{Example3}
    For this last example we consider $f=0$, $\mu=0$,
    $\epsilon=10^{-5}$,
    $\bbeta=(\cos(\frac{\pi}{3}),\sin(\frac{\pi}{3}))^{T}$, and the
    Dirichlet boundary condition $u=g$ on $\Gamma$, where $g$ is given
    by
    \begin{equation}
      g(x,y)
      =
      \left\{
      \begin{split}
        &1 \quad \text{ if } x = 0 \text{ or } y = 1
        \\
        &0  \quad \text{ otherwise.}
      \end{split}
      \right. 
    \end{equation}    
    The problem consists of propagating a discontinuous entry
    condition to the interior, thus generating an interior layer that
    meets a boundary layer at $y=1$.  We have approximated this
    problem using the meshes depicted in Figures \ref{e11}--\ref{d11}.
    For this experiment, especially the approximation of the outflow
    layer, the best results were provided by the method enhanced with
    the CIP stabilising term \eqref{eq10} and $\gamma=0.01$. So, we
    only report the results obtained for this choice.
  \end{example}
  
  For the iterative method \eqref{iter1} we use $\omega=0.1$, and we
  now report the number of fixed-point iterations needed to
  convergence:
       \begin{table}[H]
       	\centering
       	\begin{tabular}{|c|r|c|c|c|c|c|c|}
       		\hline
       		& $N$	&	5  & 9 & 17 & 33 & 65 & 129 \\ \hline
       		Mesh \ref{e11} & Itr. 	&	109  &  143& 177 & 212 & 249 & 249 \\ \hline
       		Mesh \ref{d11} & Itr. 	&123	 & 152& 186  &  218 & 245&  240 \\ \hline
       	\end{tabular}
       	\caption{Iterations needed to reach convergence using $\mathbb{P}_{1}$ elements and the meshes
       		given in Figures \ref{e11}--\ref{d11}, and the penalty term
       		\eqref{eq10} with $\gamma=0.01$ ($\omega=0.1$).}\label{tab2231}
       \end{table}
       
       \begin{table}[H]
       	\centering
       	\begin{tabular}{|c|r|c|c|c|c|c|c|}
       		\hline
       		& $N$	&	5  & 9 & 17 & 33 & 65 & 129 \\ \hline
       		$\mathbb{Q}_{1}$ & Itr. 	&156 &  226& 225 & 308 & 310 & 322 \\ \hline
       		$\mathbb{Q}_{2}$ & Itr. 	&	375 & 299&  291  &  270 &236 &  217 \\ \hline
       	\end{tabular}
       	\caption{Iterations needed to reach convergence using $\mathbb{Q}_{1}$ and $\mathbb{Q}_{2}$ elements and the mesh
       		given in Figure \ref{Q11}, and the penalty term
       		\eqref{eq121} with $\gamma_{\beta}=0.01$ ($\omega=0.1$).}\label{tab223113}
       \end{table}

  We now validate the statement made in Remark~\ref{Rem-uh-}, by
  depicting in Figure~\ref{Fig5} elevations of $u_h^-$ using $\mathbb{P}_{1}$ elements. 
  In Figure~\ref{Fig5} we depict a zoom near the boundary of the cross-section
  along the line $x=0.9$ of $u_h^-$ for different levels of refinment. We can
  observe that as the mesh gets refined,  the magnitute of $u_h^-$ decreases slowly, and the
  support of $u_h^-$ gets more and more localised, confirming what is
  stated in Remark~\ref{Rem-uh-}.
  
  Finally, in Figures \ref{fig:elevation}-\ref{fig:elevation3} we depict the approximate solutions
  using BPM, AFC, and CIP methods for this problem as well as cross
  sections showing the nature of the interior layer. We also present
  cross-sections of $u_h^+$ across $y=1-x$ for Mesh
  \ref{e11}, and $N=129$. Once again, we have performed comparisons
  with the linear CIP method (with the stabilising term \eqref{eq10}
  and $\gamma=0.01$), and the AFC scheme as described in the last
  example.  The results are depicted in Figure
  \ref{fig:elevation}-\ref{fig:elevation3}. Similar comments can be made about both sets of
  results, namely, that the current method removes the oscillations
  from the CIP solution successfully, while presenting a similar behaviour
  to AFC in terms of sharpness of the layers.


\begin{figure}[h!]
	\centering
	\subfloat[Cross-section taken of $u_h^-$ along $x = 0.9$.]{ \includegraphics[width=0.44\textwidth]{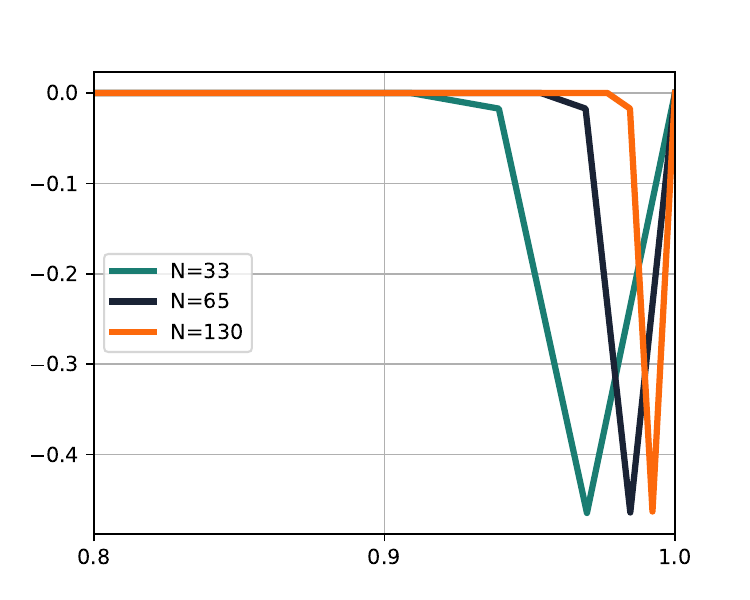}
	}
	\caption{ 
        Cross-sections of $u_h^-$ for
          Example \ref{Example3} illustrating the behaviour at the
         boundary layers using
          $\mathbb{P}_{1}$ elements and the mesh
          given in Figures \ref{e11}.}\label{Fig5}
\end{figure}

  \begin{figure}[h!]
  	\centering
  	\vspace{-10pt}
  	\subfloat[AFC approximation using mesh \ref{e11}]{
  		\includegraphics[width=0.40\textwidth]{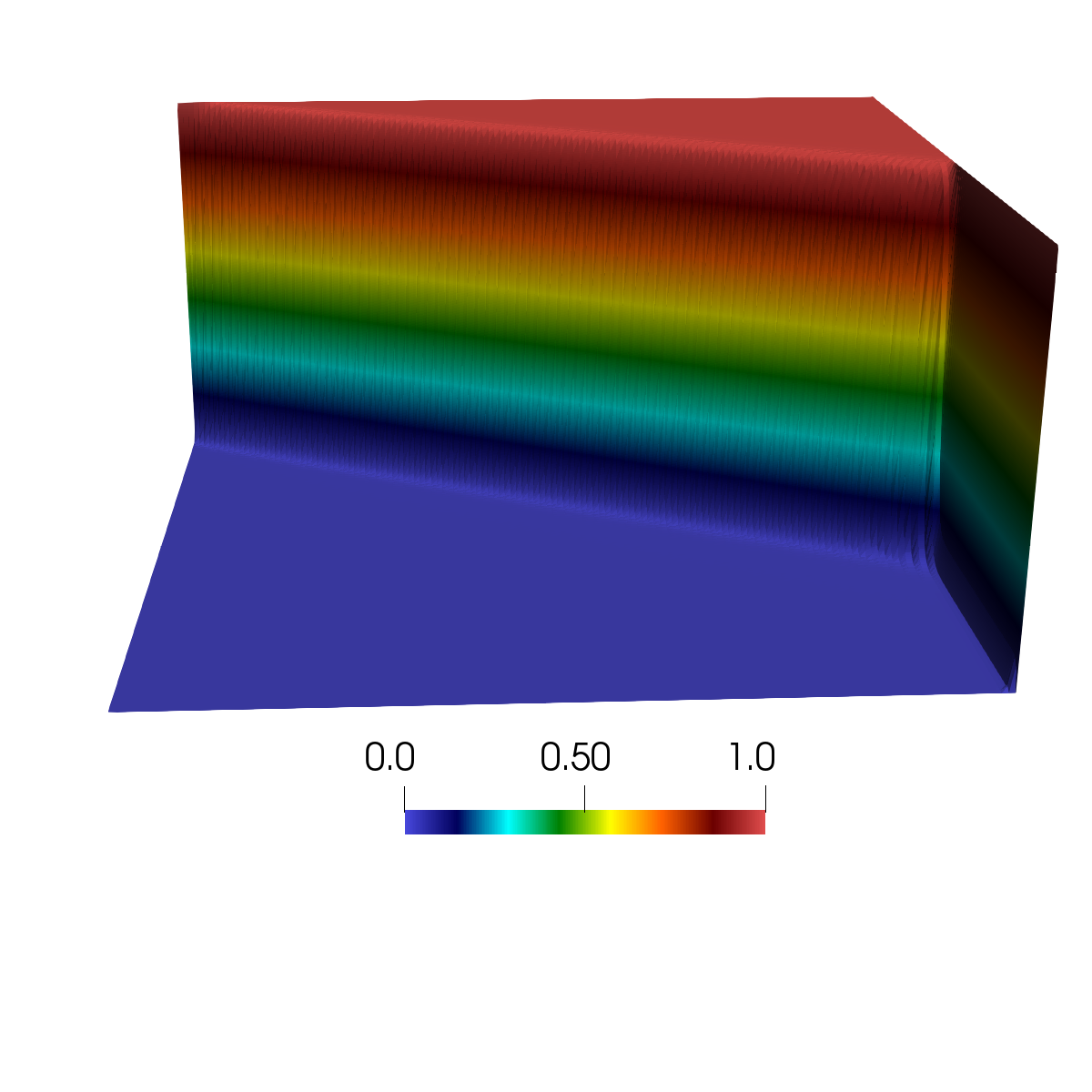}
  	}
  	\subfloat[CIP approximation using mesh \ref{e11}]{
  		\includegraphics[width=0.40\textwidth]{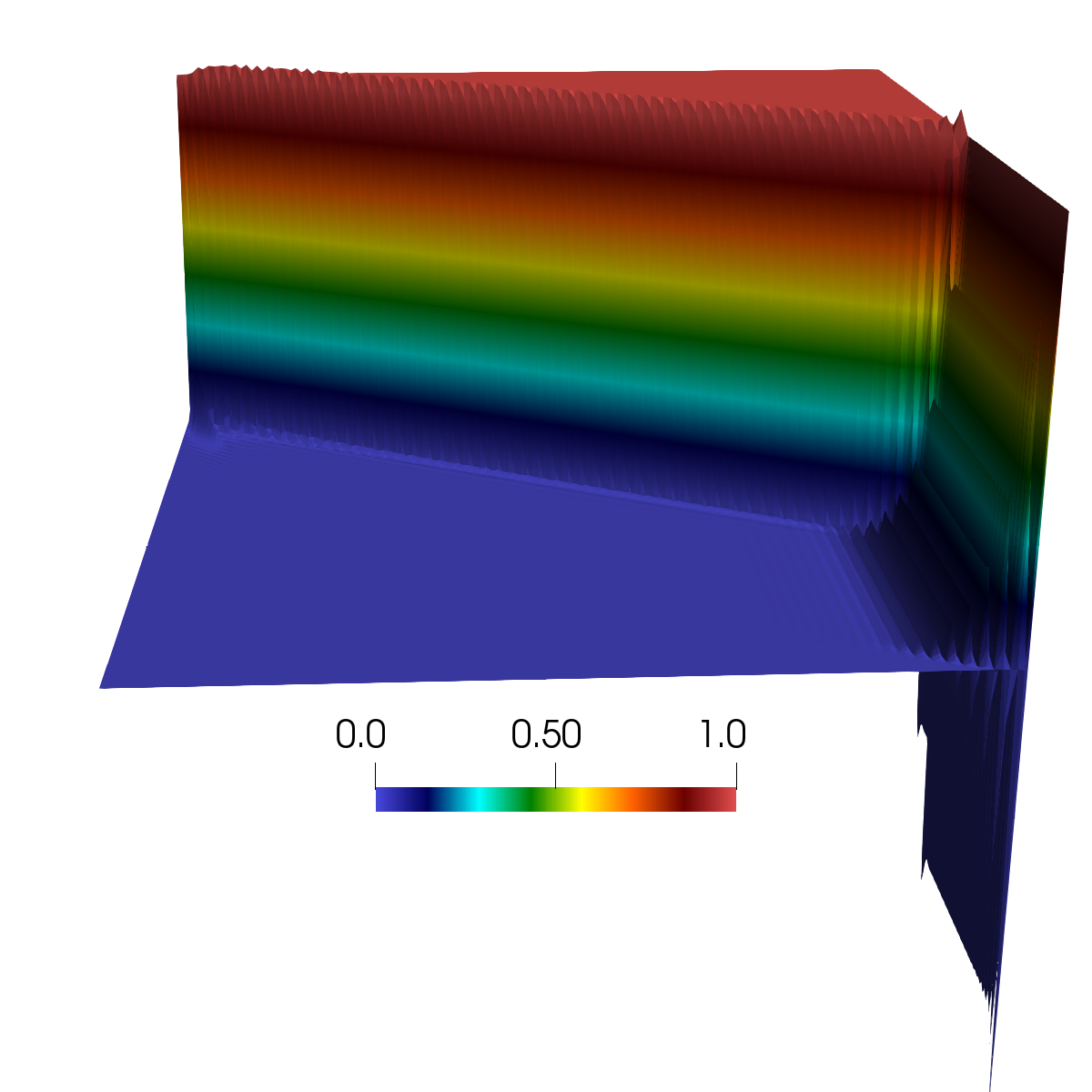}
  	}
  	\vspace{-10pt}
  	\\
  	\subfloat[BPM approximation using mesh  \ref{e11}]{
  		\includegraphics[width=0.4\textwidth]{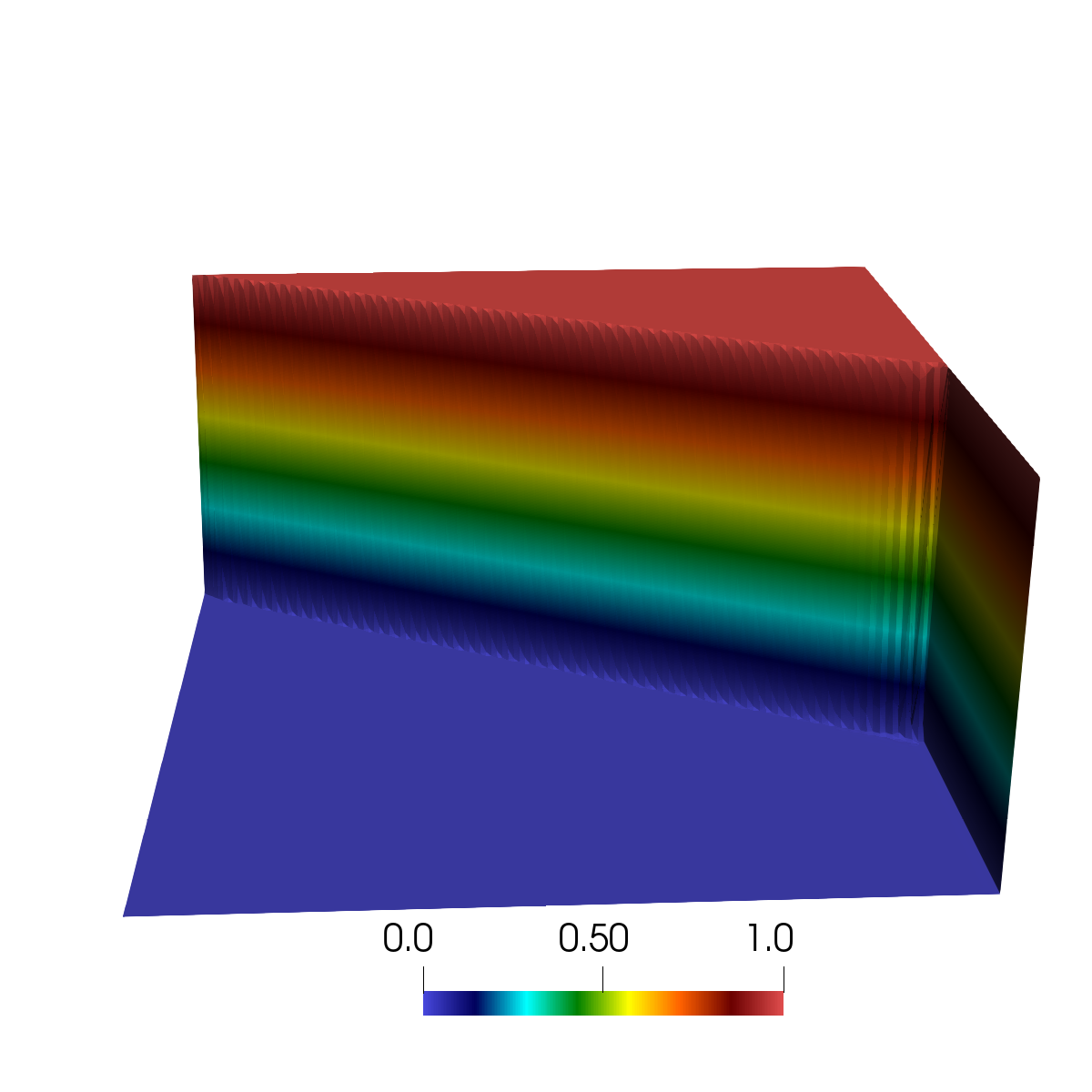}
  	}
  	\subfloat[BPM approximation using mesh \ref{d11}]{
  		\includegraphics[width=0.4\textwidth]{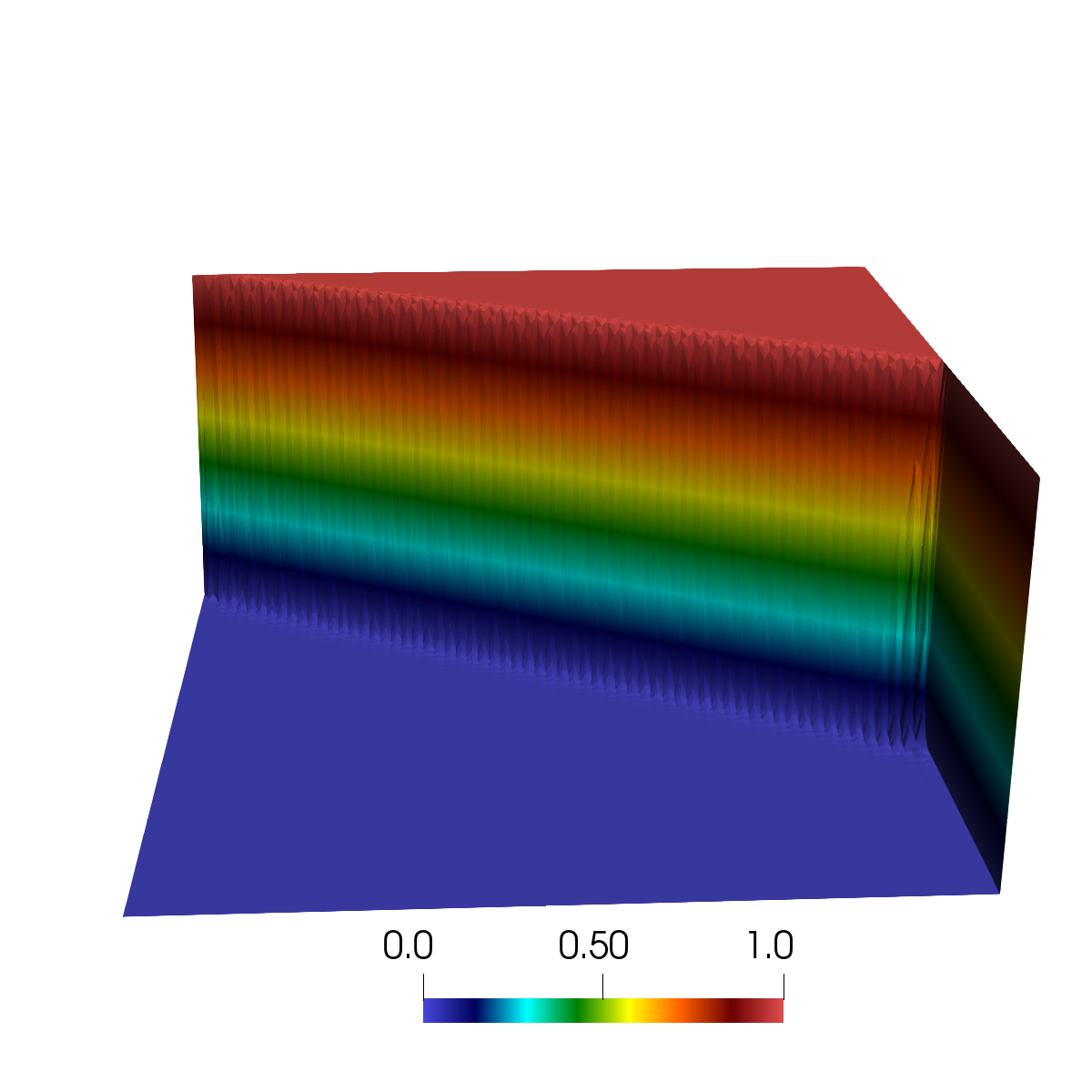}
  	}
  	\vspace{-10pt}
  	\\
  	\subfloat[Cross-section using mesh \ref{e11}]{
  		\includegraphics[width=0.4\textwidth]{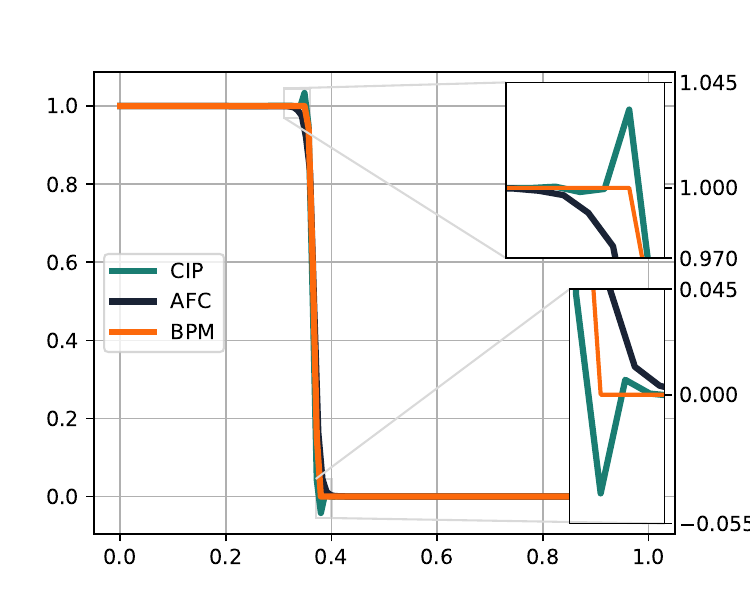}
  	}
  	\subfloat[Cross-section using mesh \ref{d11}]{
  		\includegraphics[width=0.4\textwidth]{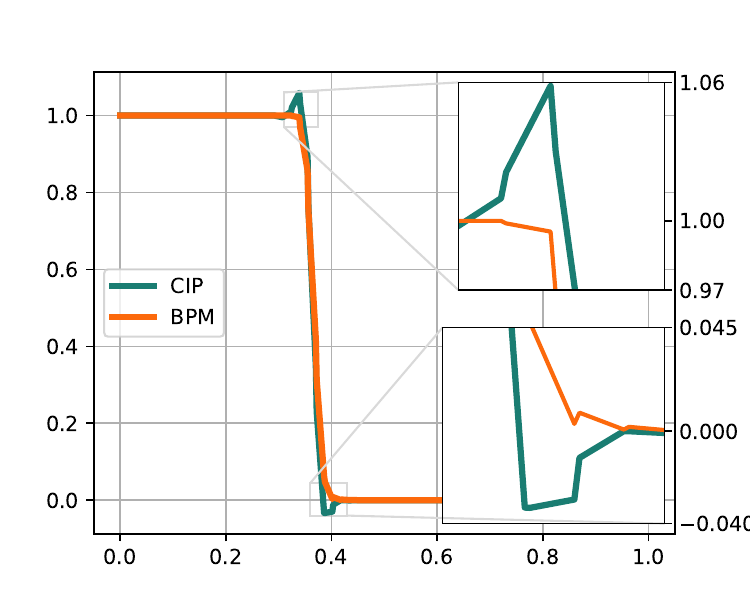}
  	}
  	\caption{The
          approximation of the solution of Example \ref{Example3} by the bound preserving method (BPM), using
          $\mathbb{P}_{1}$ elements and
          the meshes given in Figures \ref{e11} and \ref{d11} with
          $N=129$. Cross-sections taken about $y=x$ plane of the
          solution of the BPM, CIP and AFC. For AFC $p=8$ and for BPM
          and CIP the penalty (\ref{eq121}) $\gamma_{\bbeta}=0.01$ was
          used ($\omega=0.1$). For plotting the cross-sections we used linear interpolation between the nodes. }\label{fig:elevation}
  \end{figure}
  \begin{figure}[h!]
  	\centering
  	\vspace{-10pt}
  	\subfloat[CIP approximation using mesh \ref{e11}]{
  		\includegraphics[width=0.4\textwidth]{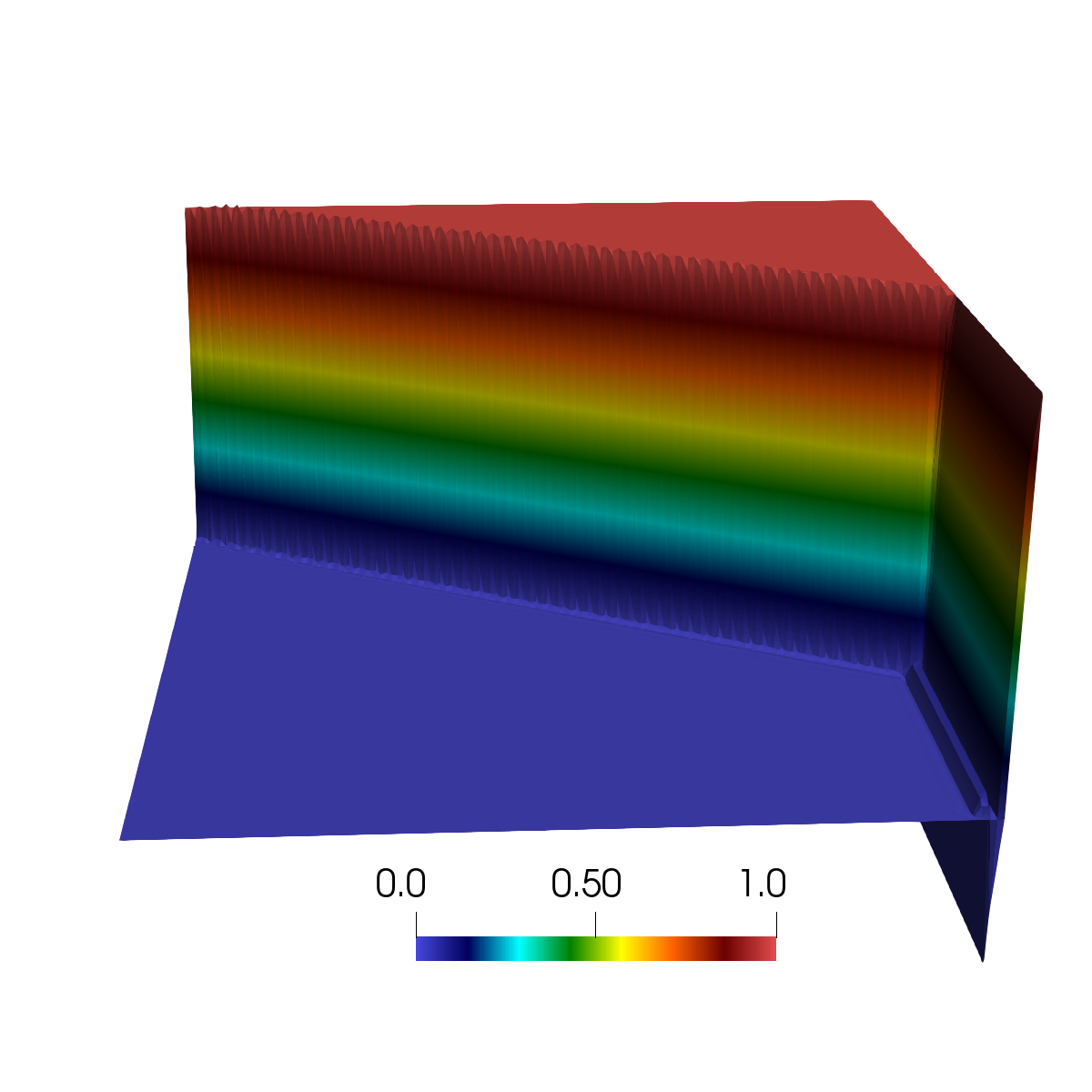}
  	}
  	\subfloat[CIP approximation using mesh \ref{d11}]{
  		\includegraphics[width=0.40\textwidth]{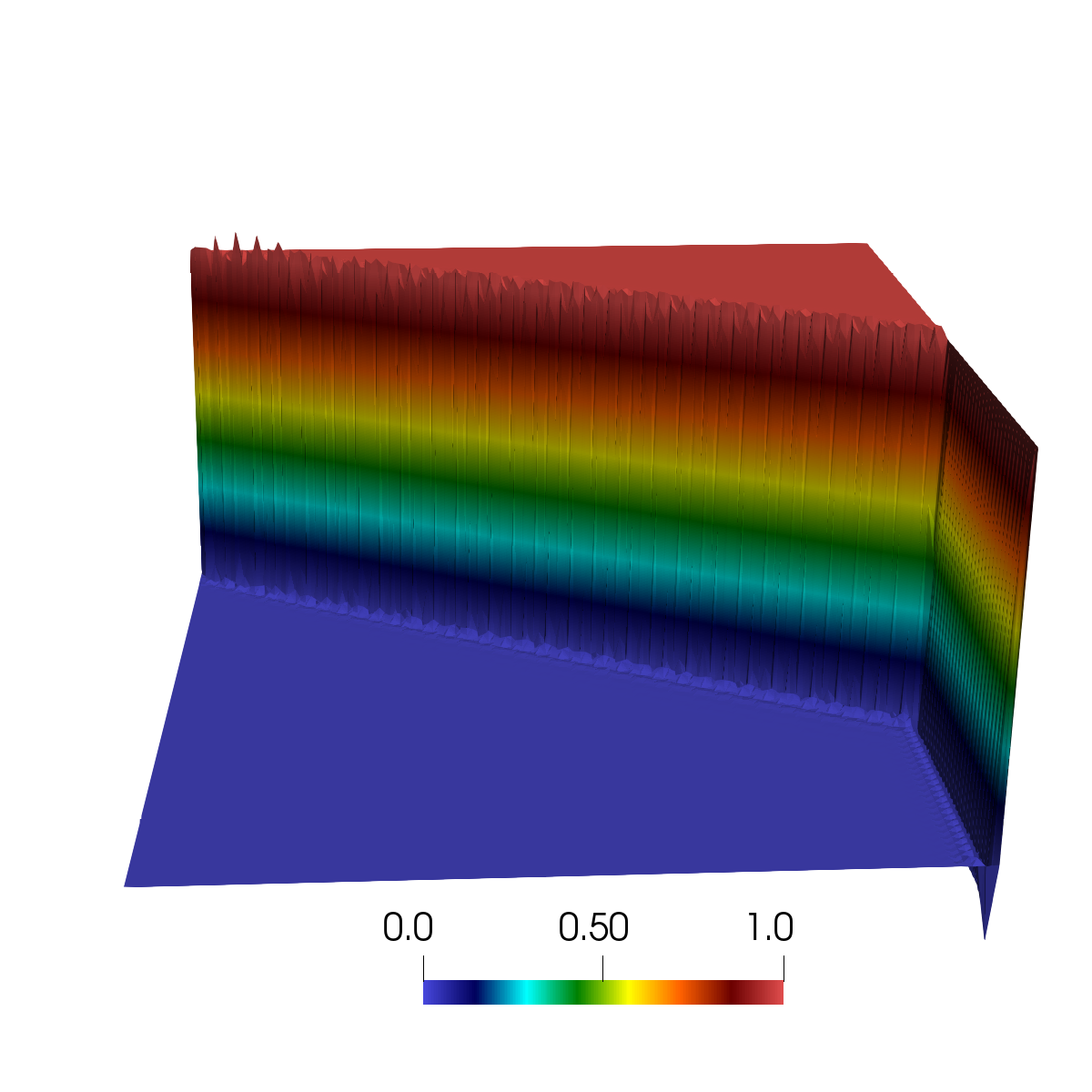}
  	}
  	\vspace{-10pt}
  	\\
  	\subfloat[BPM approximation using mesh \ref{e11}]{
  		\includegraphics[width=0.4\textwidth]{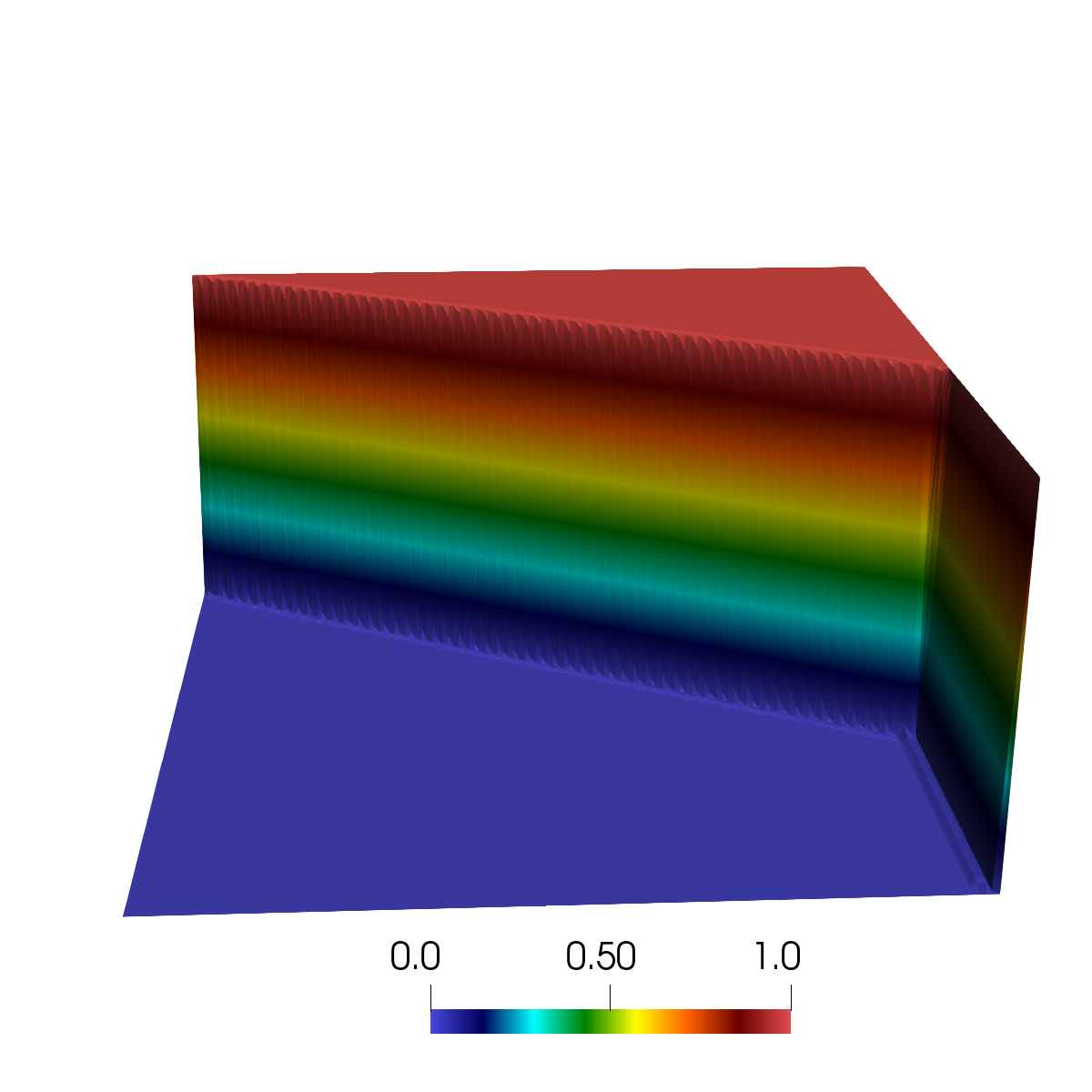}
  	}
  	\subfloat[BPM approximation using mesh \ref{d11}]{
  		\includegraphics[width=0.4\textwidth]{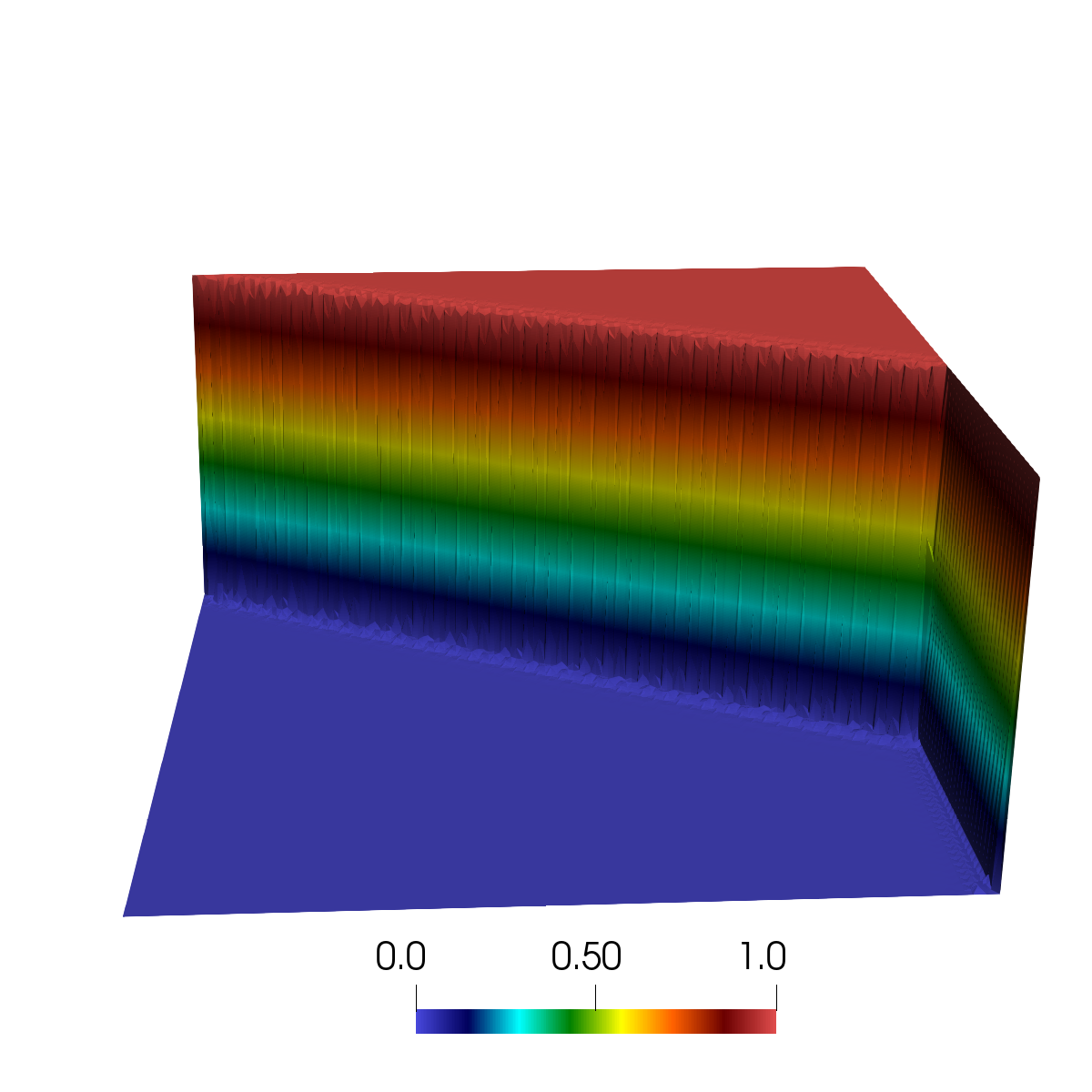}
  	}
  	\vspace{-10pt}
  	\\
  	\subfloat[Cross-section using mesh \ref{e11}]{
  		\includegraphics[width=0.4\textwidth]{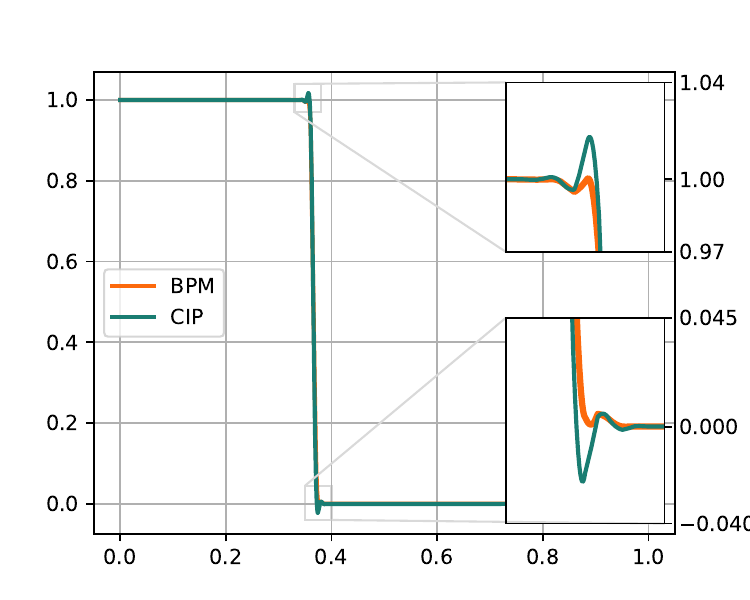}
  	}
  	\subfloat[Cross-section using mesh \ref{d11}]{
  		\includegraphics[width=0.4\textwidth]{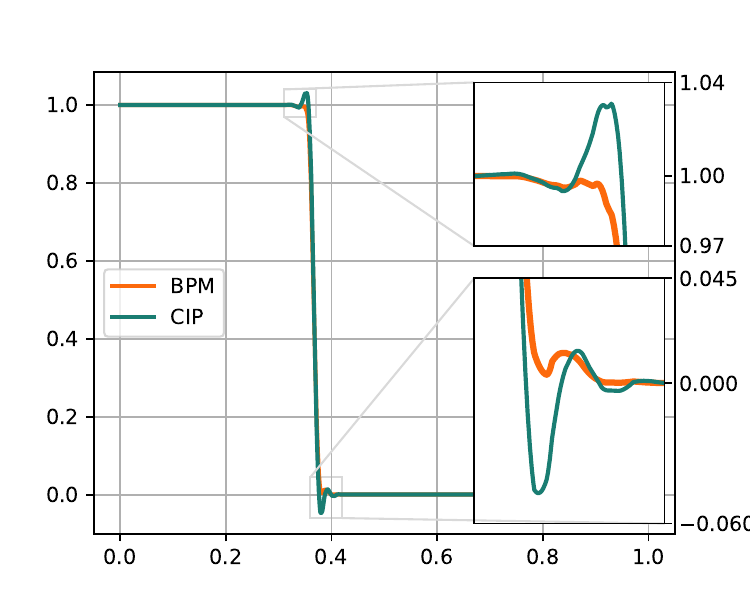}
  	}
  	\caption{The
          approximation of the solution of Example \ref{Example3} by the bound preserving method (BPM), using $\mathbb{P}_{2}$ elements
          and the meshes given in Figures \ref{e11} and \ref{d11} with
          $N=129$. Cross-sections around the line $y=x$  of the
          solution of the BPM and CIP methods. For both methods the penalty
          (\ref{eq121}) with $\gamma_{\bbeta}=0.01$ was used
          ($\omega=0.1$). 
        For plotting these cross-sections, 10,000 equidistant points were chosen along the line $y=x$, and the values of the approximated solution have been plotted at these points. }\label{fig:elevation1}
  \end{figure}

  \begin{figure}[h!]
  	\centering
  	\subfloat[CIP approximation using mesh \ref{Q11}.]{
  		\includegraphics[width=0.40\textwidth]{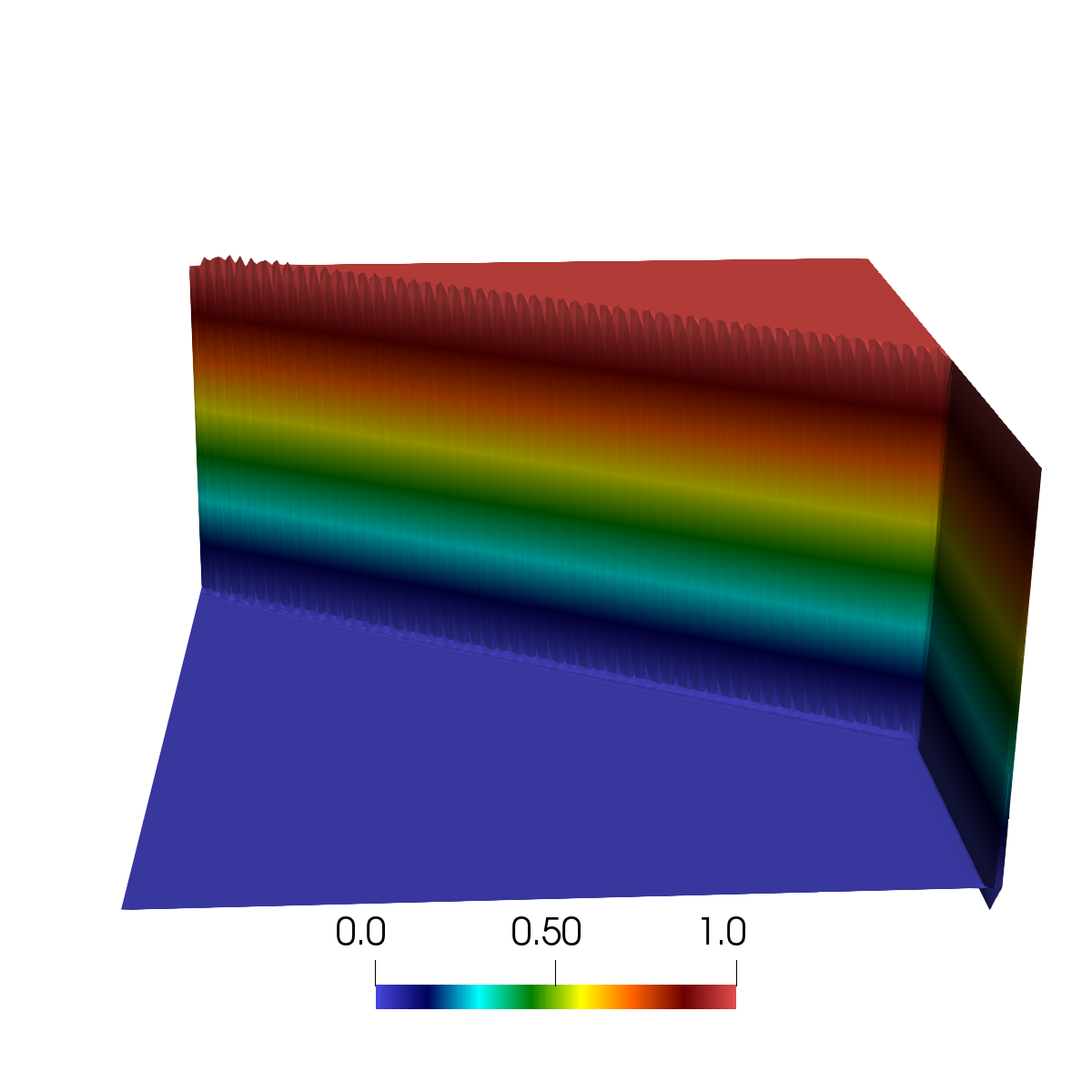}
  	}
  	\subfloat[BPM approximation using mesh \ref{Q11}.]{
  		\includegraphics[width=0.40\textwidth]{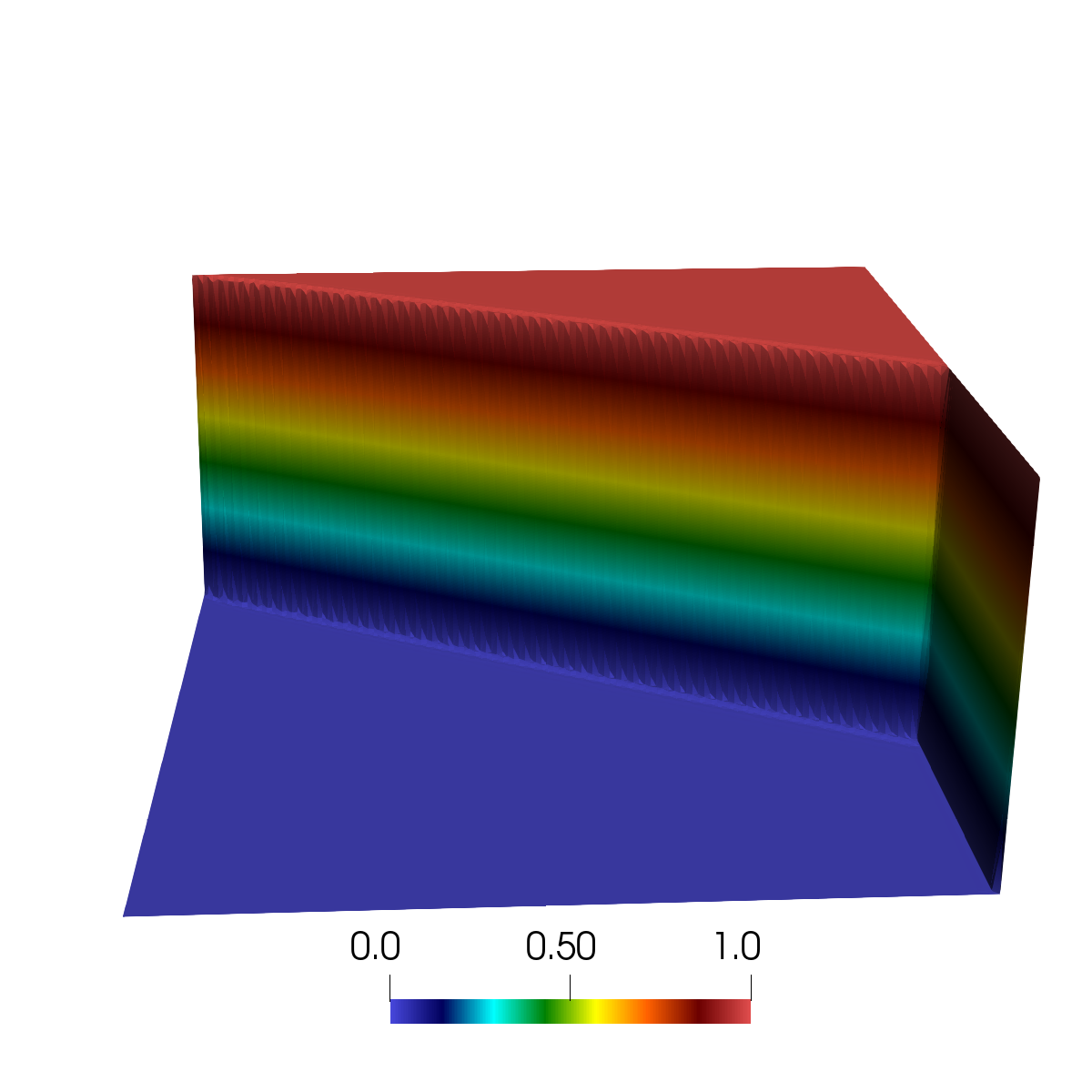}
  	}
  	\\
  	\subfloat[Cross-section taken using mesh \ref{Q11}.]{
  		\includegraphics[width=0.4\textwidth]{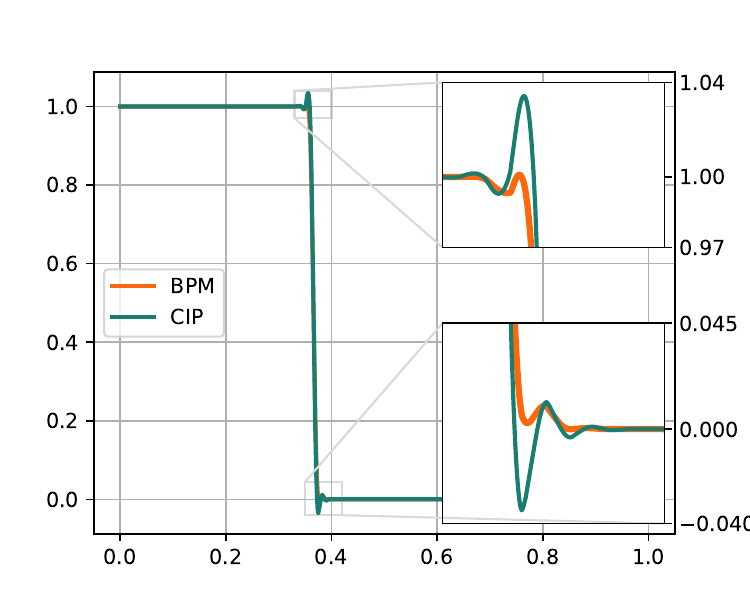}
  	}
  	\caption{The approximation of the solution of Example
  		\ref{Example3} by the bound preserving method (BPM), using $\mathbb{Q}_{2}$ elements and the
  		mesh given in Figure \ref{Q11} with
  		$N=129$. Cross-sections  of the solution
  		of the BPM and CIP taken about the line $y=x$.  For BPM and CIP the
  		penalty (\ref{eq121}) $\gamma_{\bbeta}=0.01$ was used ($\omega=0.1$). For plotting this cross-section, 10,000 equidistant points were chosen along the line $y=x$, and the values of the approximated solution have been plotted at these points.
  	}\label{fig:elevation3}
  \end{figure}